\documentclass[12pt]{amsart}

\usepackage{pinlabel}
\usepackage{amsmath,amsfonts,amssymb,latexsym,mathrsfs,stmaryrd}
\usepackage{color}
\usepackage{graphicx}
\usepackage{graphics}
\usepackage{enumerate}
\usepackage{amscd} 
\usepackage{amsxtra}
\usepackage{epic,eepic}
\usepackage{hyperref} 
\usepackage{enumitem}
\usepackage{MnSymbol}
\usepackage{dsfont}
\usepackage[all]{xy}

\newtheorem{theorem}{Theorem}[section]

\newtheorem{proposition}[theorem]{Proposition}

\newtheorem{thm}[theorem]{Theorem}
\newtheorem{lemma}[theorem]{Lemma}

\newtheorem{assumption}[theorem]{Assumption}

\theoremstyle{definition} 
\newtheorem{defn}[theorem]{Definition}
\newtheorem{definition}[theorem]{Definition}

\newtheorem{remark}[theorem]{Remark}
\newtheorem{rmk}[theorem]{Remark}

\newtheorem{conv}[theorem]{Convention}
\newcommand{\cU}{\mathcal{U}}

\newcommand{\qu}{/\kern-.7ex/}
\newcommand{\lqu}{\backslash \kern-.7ex \backslash}
\newcommand{\on}{\operatorname}

\newcommand{\Hom}{\on{Hom}}

\newcommand{\Pic}{\on{Pic}}

\newcommand{\bbox}{\on{Box}} 

\newcommand{\age}{\on{age}}

\newcommand{\HH}{\mathfrak{H}}

\title[relative GW and mirror map for toric CY]{Relative Gromov--Witten Invariants and the Enumerative meaning of Mirror Maps for Toric Calabi--Yau Orbifolds}

\author{Fenglong You}
\address{Department of Mathematical and Statistical Sciences\\ 632 CAB \\ University of Alberta\\Edmonton\\ AB\\ T6G 2G1\\ Canada}
\email{fenglong@ualberta.ca}

\thanks{}


\begin{document}
\date{\today}

\begin{abstract} 
We provide an enumerative meaning of the mirror maps for toric Calabi--Yau orbifolds in terms of relative Gromov--Witten invariants of the toric compactifications. As a consequence, we obtain an equality between relative Gromov--Witten invariants and open Gromov--Witten invariants. Therefore, the instanton corrected mirrors for toric Calabi--Yau orbifolds can be constructed using relative Gromov--Witten invariants. 
\end{abstract}

\maketitle 

\tableofcontents

\section{Introduction}

\subsection{Overview}

Gromov--Witten theory can be considered as an enumerative theory that "virtually" counts number of curves in smooth projective varieties with prescribed conditions. Gromov--Witten invariants play crucial roles in mirror symmetry. Mirror theorems can usually be stated as a relation between generating functions of genus zero Gromov--Witten invariants and period integrals of the mirrors. Mirror theorems have been proved for many targets since the proof of mirror theorem for quintic threefolds by Givental \cite{Givental96a} and Lian--Liu--Yau \cite{LLY}. Lots of structural properties, such as quantum cohomology and Givental's formalism \cite{Givental04}, have been developed.

On the other hand, relative Gromov--Witten theory provides "virtual" counts of curves with prescribed tangency conditions along a divisor. The foundations of relative Gromov--Witten theory were developed by Li--Ruan \cite{LR}, J. Li \cite{Li} and Ionel--Parker \cite{IP} about two decades ago, but several structural properties and mirror theorems had not been explored until recently. In this section, we give a brief review of some recent developments in relative Gromov--Witten theory. 

While relative Gromov--Witten invariants are difficult to compute and the structural properties were missing, orbifold Gromov--Witten invariants have been studied intensively since Chen--Ruan \cite{CR}, Abramovich--Graber--Vistoli \cite{AGV02}, \cite{AGV} and Tseng \cite{Tseng}. Lots of the structural properties and mirror theorems for Gromov--Witten theory of varieties were generalized to orbifolds. Let $X$ be a smooth projective variety and $D$ be a smooth divisor of $X$. When $r$ is a sufficiently large integer, an equality between genus zero orbifold Gromov--Witten invariants of the $r$-th root stack $X_{D,r}$ and relative Gromov--Witten invariants of $(X,D)$ was proved by Abramovich--Cadman--Wise \cite{ACW}, which generalized an earlier result of Cadman--Chen \cite{CC}. It provides a possibility of using orbifold Gromov--Witten theory to study relative Gromov--Witten theory. We will refer to such relation between relative and orbifold invariants as \emph{relative/orbifold correspondence}. The idea of using stacks to impose tangency conditions goes further back to Cadman \cite{Cadman}. However, a counterexample was provided by D. Maulik in \cite{ACW} which shows that orbifold invariants of the $r$-th root stacks do not stabilize to relative invariants in genus one when $r$ is sufficiently large. 

While the result in \cite{ACW} is inspiring, there are several further questions that need to be answered in order to study relative Gromov--Witten theory from orbifold Gromov--Witten theory. From our perspective, there are three immediate questions.

The first question is to find the relation between relative and orbifold invariants in higher genus. In joint work with H.-H. Tseng \cite{TY18a}, \cite{TY18}, we proved that orbifold invariants of the root stack $X_{D,r}$ are polynomials in $r$ when $r$ is sufficiently large. Furthermore, the constant terms of the polynomials are the corresponding relative invariants. This result indicates that it is possible to study structures of higher genus relative invariants through the structures of higher genus orbifold invariants.

The second question is how to build the structural properties for relative invariants using orbifold invariants. Since relative invariants with positive contact orders only correspond to orbifold invariants with small ages (when $r$ is sufficiently large), a generalization of relative Gromov--Witten theory is required in order to bring structural properties of orbifold Gromov--Witten theory to relative Gromov--Witten theory. In joint work with H. Fan and L. Wu, we introduced genus zero relative invariants with negative contact orders and then proved that they are equal to orbifold invariants of root stacks with large ages markings. It results in several structural properties for relative Gromov--Witten theory such as relative quantum cohomology ring, topological recursion relation (TRR), WDVV equation, Virasoro constraint (genus zero), Givental's formalism etc. An application of WDVV equation for relative Gromov--Witten theory has been worked out in \cite{FW}. In joint work with H. Fan and L. Wu \cite{FWY19}, we studied structures of higher genus relative Gromov--Witten theory and proved that relative Gromov--Witten theory forms a partial cohomological field theory.

The third question is how to obtain a mirror theorem for relative pairs. In joint work with H. Fan and H.-H. Tseng \cite{FTY}, we proved a mirror theorem for root stacks $X_{D,r}$. By taking suitable limits to the $I$-functions for $X_{D,r}$, we obtained the $I$-functions for $(X,D)$, hence a mirror theorem for relative pairs $(X,D)$. Passing from orbifold mirror theorems to relative mirror theorems relies on the genus zero relative/orbifold correspondence in \cite{ACW}, \cite{TY18} and \cite{FWY}.

\subsection{Motivation}

In this paper, we turn our attention to possible connections among relative Gromov--Witten theory, mirror maps, open Gromov--Witten theory and mirror constructions. We will focus on toric Calabi--Yau orbifolds.

For a Calabi--Yau manifold $X$ and its mirror $\check{X}$, the Strominger-Yau-Zaslow (SYZ) conjecture \cite{SYZ} asserts the existence of special Lagrangian torus fibrations $\mu:X\rightarrow B$ and $\check{\mu}:\check{X}\rightarrow B$ which are fiberwise-dual to each other. SYZ mirror symmetry for toric Calabi--Yau manifolds was studied in \cite{CLL} and generalized to toric Calabi--Yau orbifolds in \cite{CCLT}. The SYZ mirror of $X$ needs to be modified by the instanton correction. 

Following Auroux \cite{Auroux07},\cite{Auroux09}, the instanton-corrected mirror of a toric Calabi--Yau orbifold $\mathcal X$ was constructed in \cite{CCLT} where the instanton corrections are given by genus zero open Gromov--Witten invariants which can be considered as virtual counts of holomorphic orbi-disks in $\mathcal X$ bounded by fibers of the Gross fibration. As an application, the SYZ map, defined in \cite{CCLT} in terms of genus zero open Gromov--Witten invariants, is equal to the inverse of the mirror map. Such an enumerative meaning of the inverse of the mirror map was proved earlier in \cite{CLT} for the total space of the canonical bundle of a compact toric Fano variety. Such a relation between disk countings and mirror maps was first observed by Gross--Siebert in \cite[Conjecture 0.2]{GS11} for tropical disks. Furthermore, the generating functions of open Gromov--Witten invariants for toric Calabi--Yau manifolds in \cite{CCLT} were proved to be the same as the slab functions in Gross--Siebert program. The slab functions were originally interpreted as counting tropical disks in \cite{GS14}. A possible alternative approach, originates in \cite{GPS}, via logarithmic Gromov--Witten theory was also sketched in \cite{GS14}. The periods for local Calabi--Yau varieties have also been computed in \cite{RS}.

In this paper, we study an alternative way of counting holomorphic disks using relative Gromov--Witten invariants. More specifically, given a toric Calabi--Yau orbifold $\mathcal X$, we consider the pair $(\bar{\mathcal X},D_\infty)$, where $\bar{\mathcal X}$ is a toric compactification of $\mathcal X$ and $D_\infty:=\bar{\mathcal X}\setminus \mathcal X$ is a toric prime divisor of the compactification $\bar{\mathcal X}$. We consider relative Gromov--Witten invariants of $(\bar{\mathcal X},D_\infty)$ with one relative marking. They can be considered as virtual counts of curves in $\bar{\mathcal X}$ meeting the divisor $D_\infty$ at one point. Hence they can be considered as holomorphic disk counts of $\mathcal X=\bar{\mathcal X}\setminus D_\infty$ which is the complement of the divisor $D_\infty$ in $\bar{\mathcal X}$. The idea of using relative/log Gromov--Witten invariants as an algebro-geometric version of disk countings can be seen in several literatures, see, for example, \cite{GHK}. 

\subsection{Main results}

Let $\beta^\prime\in \pi_2(\mathcal X,L)$ be a basic (orbi-)disk class with Chern--Weil Maslov index $2$, where $L$ is a Lagrangian torus fiber of the moment map of the toric Calabi--Yau orbifold $\mathcal X$. The genus zero open Gromov--Witten invariant
\[
n_{1,l,\beta}^{\mathcal X}([\on{pt}]_L;{\textbf 1}_{v_1},\ldots,{\textbf 1}_{v_l})
\]
is defined in \cite[Definition 3.5]{CCLT}, where $\beta=\beta^\prime+\alpha$ for some $\alpha\in H_2^{\on{eff}}(\mathcal X)$. 

On the other hand, there is a one-to-one correspondence between the basic (orbi-)disk classes and the rays of the stacky fan of the toric Calabi--Yau orbifold $\mathcal X$. The (orbi-)disk class $\beta^\prime$ can be used to construct a toric compactification $\bar{\mathcal X}$ of $\mathcal X$ by adding an extra ray (the negative of the ray corresponding to $\beta^\prime$) to the stacky fan, see Section \ref{sec:compactification} and \cite[Section 6.1]{CCLT} for details. The toric compactification $\bar{\mathcal X}$ depends on the choice of $\beta^\prime$. The compactification introduces an extra toric prime divisor $D_\infty:=\bar{\mathcal X}\setminus \mathcal X$. We consider the following generating functions for genus zero relative Gromov--Witten invariants:
\begin{align}
1+\delta_i^{\on{rel}}:=\sum_{\alpha\in H_2^{\on{eff}}(\mathcal X)}\sum_{l\geq 0}\sum_{v_1,\ldots v_l\in \bbox(\Sigma)^{\age=1}}\frac{\prod_{i=1}^l \tau_{v_i}}{l!}\left\langle [\on{pt}]_{D_\infty},\textbf{1}_{\bar{v}_1},\ldots,\textbf{1}_{\bar{v}_l}\right\rangle_{0,1+l, \bar{\beta}^\prime+\alpha}^{(\bar{\mathcal X},D_\infty)}q^{\alpha},
\end{align}
where $\bar{\beta}^\prime:=\beta^\prime+\beta_\infty\in H_2(\bar{\mathcal X},\mathbb Q)$, $\beta^\prime=\beta_{i}$ is a basic smooth disk class corresponding to the primitive generator $b_i$ of  a ray in the stacky fan $\Sigma$ of $\bar{\mathcal X}$, and $\beta_\infty$ is the basic smooth disk class corresponding to the primitive generator $b_\infty:=-b_i$ of the extra ray;
\begin{align}
\tau_v+\delta_v^{\on{rel}}:=\sum_{\alpha\in H_2^{\on{eff}}(\mathcal X)}\sum_{l\geq 0}\sum_{v_1,\ldots v_l\in \bbox(\Sigma)^{\age=1}}\frac{\prod_{i=1}^l \tau_{v_i}}{l!}\left\langle [\on{pt}]_{D_\infty},\textbf{1}_{\bar{v}_1},\ldots,\textbf{1}_{\bar{v}_l}\right\rangle_{0,1+l, \bar{\beta}^\prime+\alpha}^{(\bar{\mathcal X},D_\infty)}q^{\alpha},
\end{align}
where $\bar{\beta}^\prime:=\beta^\prime+\beta_\infty$ and $\beta^\prime=\beta_{v}$ is a basic orbi-disk class corresponding to a box element $v\in \bbox(\Sigma)^{\age=1}$.

The computation of these invariants relies on the mirror theorem for relative pairs $(\bar{\mathcal X},D_\infty)$. Note that the mirror theorem for a pair $(X,D)$ in \cite{FTY} requires $X$ to be a smooth projective variety and $D$ to be a smooth nef divisor. We need a generalization of the result in \cite{FTY} to orbifold pairs in order to compute relative invariants of the toric orbifolds $(\bar{\mathcal X},D_\infty)$. We briefly sketch the genus zero relative/orbifold correspondence with stacky targets in Appendix \ref{sec:stacky-rel=orb}. In \cite{FTY}, the proof of the mirror theorem for pairs relies on orbifold quantum Lefschetz principle in \cite{Tseng} and the mirror theorem for toric stack bundles in \cite{JTY}. Since orbifold quantum Lefschetz principle can fail for positive orbifold hypersurfaces \cite{CGIJJM} and the mirror theorem in \cite{JTY} requires the toric stack bundles to have a smooth base, the result in Appendix \ref{sec:stacky-rel=orb} does not imply a generalization of the mirror theorem for general orbifold pairs. However, as we are working on the toric Calabi--Yau orbifold $\bar{\mathcal X}$ along with the toric divisor $D_\infty$, the root stack $\bar{\mathcal X}_{D_\infty, r}$ is still a toric orbifold. We can avoid the hypersurface construction in \cite{FTY} and simply apply the mirror theorem for toric stacks in \cite{CCIT}. By taking a suitable limit as in \cite[Section 4]{FTY}, we obtain a mirror theorem for the pair $(\bar{\mathcal X},D_\infty)$, see Theorem \ref{thm:mirror-toric-orbi-pair}.

Using mirror theorem for the pair $(\bar{\mathcal X},D_\infty)$, we obtain explicit formulas for the generating functions of relative invariants of $(\bar{\mathcal X},D_\infty)$.

\begin{theorem}[=Theorem \ref{thm:rel-smooth}]
If $\beta^\prime=\beta_{i_0}$ is a basic smooth disk class corresponding to the primitive generator $b_{i_0}$ of a ray in the stacky fan $\Sigma$ of $\bar{\mathcal X}$, then we have
\begin{align*}
1+\delta_{i_0}^{\on{rel}}:=\exp(-g_{i_0}(y(q,\tau))),
\end{align*}
where $g_{i_0}(y)$ is defined in (\ref{g-j-smooth}); $y=y(q,\tau)$ is the inverse of the toric mirror map of $\mathcal X$ in Proposition \ref{prop:toric-mirror-map}.
\end{theorem}

\begin{theorem}[=Theorem \ref{thm:rel-orbi}]
If $\beta^\prime=\beta_{v_{j_0}}$ is a basic orbi-disk class corresponding to a box element $v_{j_0}\in \bbox(\Sigma)^{\age=1}$, then we have
\begin{align*}
\tau_{v_{j_0}}+\delta_{v_{j_0}}^{\on{rel}}:= y^{D_{j_0}^\vee}\exp\left(-\sum_{i\not\in I_{j_0}}c_{j_0 i}g_{i}(y(q,\tau))\right),
\end{align*}
where $g_{i}(y)$ is defined in (\ref{g-j-orbi}); $y=y(q,\tau)$ is the inverse of the toric mirror map of $\mathcal X$ in Proposition \ref{prop:toric-mirror-map}; $D_{j_0}^\vee$ is the class defined in (\ref{D-dual}); $I_{j_0}\in \mathcal A$ is the anticone of the minimal cone containing $b_{j_0}=\sum_{i\not\in I_{j_0}}c_{j_0i}b_i$.

\end{theorem}

The above formulas for relative potentials coincide with the formulas for open potentials. As a consequence, we have the following equality between open invariants and relative invariants. Hence, these two ways of counting disks coincide. 

\begin{theorem}[=Theorem \ref{thm:open=rel}]
The following open invariants and relative invariants are equal:
\[
n_{1,l,\beta^\prime+\alpha}^{\mathcal X}( [\on{pt}]_L,\textbf{1}_{v_1},\ldots,\textbf{1}_{v_l})=\left\langle [\on{pt}]_{D_\infty},\textbf{1}_{\bar{v}_1},\ldots,\textbf{1}_{\bar{v}_l}\right\rangle_{0,1+l, \bar{\beta}^\prime+\alpha}^{(\bar{\mathcal X},D_\infty)},
\]
where $\bar{\beta}^\prime=\beta^\prime+\beta_\infty$ and $\beta^\prime\in \pi_2(\mathcal X,L)$ is either a smooth disk class or an orbi-disk class.
\end{theorem}

Therefore, the SYZ mirror of the toric Calabi--Yau orbifold $\mathcal X$ can be constructed using relative Gromov--Witten invariants instead of open Gromov--Witten invariants.

\begin{theorem}[=Theorem \ref{thm:SYZ-mirror}]
Let $\mathcal X$ be a toric Calabi--Yau orbifold equipped with the Gross fibration, the SYZ mirror of $\mathcal X$ (with a hypersurface removed) is the family of non-compact Calabi--Yau
\[
\check{\mathcal X}_{q,\tau}=\{(u,v,z_1,\ldots,z_{n-1})\in \mathbb C^2\times (\mathbb C^\times)^{n-1}| uv=G_{(q,\tau)}(z_1,\ldots,z_{n-1})\},
\]
where
\[
G_{(q,\tau)}(z_1,\ldots,z_{n-1})=\sum_{i=0}^{m-1}C_i(1+\delta_i^{\on{rel}})z^{b_i}+\sum_{j=m}^{m^\prime-1}C_{v_j}(\tau_{v_j}+\delta_{v_j}^{\on{rel}})z^{v_j},
\]

The SYZ mirror of $\mathcal X$ without removing a hypersurface is given by a Landau-Ginzburg model $(\check{\mathcal X},W)$, where $W:\check{\mathcal X}\rightarrow \mathbb C$ is the holomorphic function $W:=u$.
\end{theorem}
It was proved in \cite{Lau} that Gross--Siebert's slab functions and generating functions for open invariants coincide for toric Calabi--Yau manifolds. The enumerative meaning of the slab functions is given in term of tropical disks in \cite{GS14}. Hence, we conclude that the three ways (tropical, symplectic and algebraic) of counting disks coincide.

The rest of the paper is organized as follows. In Section \ref{sec:rel-GW}, we give a brief review of the definition of relative Gromov--Witten invariants, the genus zero relative/orbifold correspondence and the mirror theorem for relative pairs. In Section \ref{sec:toric}, we first give a brief review of the definition of toric orbifolds and mirror theorem for toric orbifolds. We also state mirror theorem for orbifold toric pairs which will be used to compute relative invariants in Section \ref{sec:compute-rel-inv}. In Section \ref{sec:toric}, we also give a brief review of open Gromov--Witten invariants of toric orbifolds and toric compactifications of toric Calabi--Yau orbifolds. In Section \ref{sec:compute-rel-inv}, we compute relative invariants of toric compatifications of toric Calabi--Yau orbifolds and prove that relative invariants provide an enumerative meaning of the inverse of the mirror map. As consequences, we show that open invariants are equal to relative invariants; and relative invariants can be used to construct the instanton corrected mirror of toric Calabi--Yau orbifolds. In Appendix \ref{sec:stacky-rel=orb}, we briefly sketch a proof of the genus zero relative/orbifold correspondence with stacky targets.

\section*{Acknowledgment}
The author is very grateful to Hsian-Hua Tseng for related collaborations and for his explanations of his previous work \cite{CCLT}. The author would also like to thank Denis Auroux, Honglu Fan, Michel Van Garrel, Yu-Shen Lin, Melissa Liu and Jingyu Zhao for related discussions. This project is supported by the postdoctoral fellowship of NSERC and Department of Mathematical Sciences at the University of Alberta.

\section{Relative Gromov--Witten theory}\label{sec:rel-GW}

\subsection{Definition}

We begin with a brief review of the definition of relative Gromov--Witten invariants. We refer to \cite{LR}, \cite{Li01}, \cite{Li}, and \cite{IP} for more details about the construction of relative Gromov--Witten theory. 

Let $X$ be a smooth projective variety and $D$ be a smooth divisor. 
For $d\in H_2(X,\mathbb Q)$, we consider a partition $\vec k=(k_1,\ldots,k_m)$ of $\int_d[D]$. That is,
\[
\sum_{i=1}^m k_i=\int_d[D], \quad k_i\in \mathbb Z_{>0}, \text{ for } i\in \{1,\ldots, m\}.
\] 
We consider the moduli space $\overline{M}_{g,\vec k,n,d}(X,D)$ of $(m+n)$-pointed, genus $g$, degree $d\in H_2(X,\mathbb Q)$, relative stable maps to $(X,D)$ such that the first $m$ marked points are relative marked points whose contact orders with the divisor $D$ are given by the partition $\vec k$ and the last $n$ marked points are interior marked points. There are two types of evaluation maps corresponding to two types of markings. Let $\on{ev}_i$ be the $i$-th evaluation map, then
\begin{align*}
\on{ev}_i: \overline{M}_{g,\vec k,n,d}(X,D) \rightarrow D, & \quad\text{for } 1\leq i\leq m;\\
\on{ev}_i: \overline{M}_{g,\vec k,n,d}(X,D) \rightarrow X, & \quad \text{for } m+1\leq i\leq m+n.
\end{align*}

There is a stabilization map $s:\overline{M}_{g,\vec k,n,d}(X,D)\rightarrow \overline{M}_{g,m+n,d}(X)$. Let $\bar{\psi}_i=s^*(\psi_i)$ be the descendant class at the $i$-th marked point which is the pullback of the descendant class from $\overline{M}_{g,m+n,d}(X)$.
Consider
\begin{itemize}
\item $\delta_{i}\in H^*(D,\mathbb Q)$, for $1\leq i\leq m$.
\item $\gamma_{m+i}\in H^*(X,\mathbb Q)$, for $1\leq i\leq n$.
\item  $a_{i}\in \mathbb Z_{\geq 0}$, for $1\leq i\leq m+n$.
\end{itemize}
Relative Gromov--Witten invariants of $(X,D)$ are defined as
\begin{align}\label{relative-invariant-higher-dimension}
&\left\langle \prod_{i=1}^m\tau_{a_{i}}(\delta_{i})\left|\prod_{i=1}^n \tau_{a_{m+i}}(\gamma_{m+i})\right.\right\rangle^{(X,D)}_{g,\vec k,n,d}:=\\
\notag &\int_{[\overline{M}_{g,\vec k,n,d}(X,D)]^{vir}}\bar{\psi}_1^{a_1}\on{ev}^*_{1}(\delta_{1})\cdots \bar{\psi}_m^{a_m}\on{ev}^*_{m}(\delta_{m})\bar{\psi}_{m+1}^{a_{m+1}}\on{ev}^{*}_{m+1}(\gamma_{m+1})\cdots\bar{\psi}_{m+n}^{a_{m+n}}\on{ev}^{*}_{m+n}(\gamma_{m+n}).
\end{align}
 When $\mathcal X$ is a smooth proper Deligne--Mumford stack and $\mathcal D\subset \mathcal X$ is a smooth irreducible divisor, relative Gromov--Witten invariants of orbifold pairs $(\mathcal X,\mathcal D)$ can be defined in a similar way.

\subsection{Relative/orbifold correspondence}

The invariants that are naturally related to the relative invariants of $(X,D)$ are the orbifold invariants of the $r$-th root stack $X_{D,r}$ for a sufficiently large integer $r$. The relationship between them has been studied in \cite{ACW},\cite{TY18a}, \cite{TY18} and \cite{FWY}.

The evaluation maps for orbifold Gromov-Witten invariants land on the inertia stack of the target orbifold. The coarse moduli space of the inertia stack of the root stack $X_{D,r}$ can be decomposed into disjoint union of $r$ components
\[
\underline{I}(X_{D,r})=X\sqcup_{i=1}^{r-1} D,
\]
where there are $r-1$ components, called twisted sectors, isomorphic to $D$. The twisted sectors $\mathcal D_r$ of the inertial stack $IX_{D,r}$ are $\mu_r$-gerbes over $D$. Twisted sectors are labeled by rational numbers called ages. 

The relative conditions can be translated to orbifold conditions as follows, we consider the moduli space $\overline{M}_{g,\vec k,n,d}(X_{D,r})$ of $(m+n)$-pointed, genus g, degree $d$, orbifold stable maps to $X_{D,r}$ whose orbifold conditions are given by the partition $\vec k$ as follows.
\begin{itemize}
\item for $1\leq i\leq m$, the coarse evaluation map $\on{ev}_i$ at the $i$-th marked point lands on the twisted sector $D$ with age $k_i/r$. These marked points are orbifold marked points.
\item the coarse evaluation maps $\on{ev}_i$ at the last $n$ marked points all land on the identity component $X$ of the coarse moduli space of the inertia stack $I(X_{D,r})$. These marked points are interior marked points.
\end{itemize}

Orbifold Gromov--Witten invariants of $X_{D,r}$ are defined as

\begin{align}\label{orbifold-invariant-higher-dimension}
&\left\langle \prod_{i=1}^m\tau_{a_{i}}(\delta_{i})\prod_{i=1}^n \tau_{a_{m+i}}(\gamma_{m+i})\right\rangle^{X_{D,r}}_{g,\vec k,n,d}:=\\
\notag &\int_{[\overline{M}_{g,\vec k,n,d}(X_{D,r})]^{vir}}\bar{\psi}_1^{a_1}\on{ev}^*_{1}(\delta_{1})\cdots \bar{\psi}_m^{a_m}\on{ev}^*_{m}(\delta_{m})\bar{\psi}_{m+1}^{a_{m+1}}\on{ev}^{*}_{m+1}(\gamma_{m+1})\cdots\bar{\psi}_{m+n}^{a_{m+n}}\on{ev}^{*}_{m+n}(\gamma_{m+n}),
\end{align}
where the descendant class $\bar{\psi}_i$ is the class pullback from the corresponding descendant class on the moduli space $\overline{M}_{g,m+n,d}(X)$ of stable maps to $X$.

We refer to \cite{Abramovich}, \cite{AGV02}, \cite{AGV}, \cite{CR} and \cite{Tseng} for more details about orbifold Gromov--Witten invariants.

\begin{remark}
We will use the notation 
\[
\left\langle \delta_1\bar{\psi}^{a_1},\ldots, \delta_m\bar{\psi}^{a_m}, \gamma_{m+1}\bar{\psi}^{a_{m+1}},\ldots, \gamma_{m+n}\bar{\psi}^{a_{m+n}}\right\rangle^{X_{D,r}}_{g,m+n,d}
\]
instead of the notation in (\ref{orbifold-invariant-higher-dimension}) to denote orbifold invariants of $X_{D,r}$ when orbifold conditions are clear from the context. Similarly for relative invariants of $(X,D)$, we may use 
\[
\left\langle \delta_1\bar{\psi}^{a_1},\ldots, \delta_m\bar{\psi}^{a_m}, \gamma_{m+1}\bar{\psi}^{a_{m+1}},\ldots, \gamma_{m+n}\bar{\psi}^{a_{m+n}}\right\rangle^{(X,D)}_{g,m+n,d}
\] 
instead of the notation in (\ref{relative-invariant-higher-dimension}) if relative conditions are clear. 
\end{remark}

The relationship between genus zero relative invariants of $(X,D)$ and orbifold invariants of $X_{D,r}$ is proved in \cite{ACW}. The relationship between higher genus relative and orbifold invariants is proved in \cite{TY18a} and \cite{TY18}. For the purpose of this paper, we only state the genus zero case. 

\begin{theorem}[\cite{ACW}, Theorem 1.2.1; \cite{TY18}, Theorem 5.1]\label{thm:rel=orb}
Let $X$ be a smooth projective variety and $D$ a smooth divisor. For a sufficiently large $r$, genus zero relative and orbifold invariants coincide:
\[
\left\langle \prod_{i=1}^m\tau_{a_{i}}(\delta_{i})\left|\prod_{i=1}^n \tau_{a_{m+i}}(\gamma_{m+i})\right.\right\rangle^{(X,D)}_{0,\vec k,n,d}=\left\langle \prod_{i=1}^m\tau_{a_{i}}(\delta_{i})\prod_{i=1}^n \tau_{a_{m+i}}(\gamma_{m+i})\right\rangle^{X_{D,r}}_{0,\vec k,n,d}.
\]
\end{theorem}

Since we will study relative invariants of toric compactification of toric Calabi--Yau orbifolds, we will need to generalize the result of Theorem \ref{thm:rel=orb} to the case when $X$ is a smooth Deligne--Mumford stack. 

\begin{thm}\label{thm:stacky-rel=orb}
Let $\mathcal X$ be a smooth proper Deligne--Mumford stack and $\mathcal D\subset \mathcal X$ a smooth irreducible divisor. Genus zero relative and orbifold invariants coincide when $r$ is sufficiently large. That is 
\[
\left\langle \prod_{i=1}^m\tau_{a_{i}}(\delta_{i})\left|\prod_{i=1}^n \tau_{a_{m+i}}(\gamma_{m+i})\right.\right\rangle^{(\mathcal X,\mathcal D)}_{0,\vec k,n,d}=\left\langle \prod_{i=1}^m\tau_{a_{i}}(\delta_{i})\prod_{i=1}^n \tau_{a_{m+i}}(\gamma_{m+i})\right\rangle^{\mathcal X_{\mathcal D,r}}_{0,\vec k,n,d}.
\]
\end{thm}

The proof of Theorem \ref{thm:stacky-rel=orb} follows from the ideas in \cite{TY18}. In Appendix \ref{sec:stacky-rel=orb}, we provide a sketch of the proof of Theorem \ref{thm:stacky-rel=orb}. More details of the proof will appear in the a forthcoming work with H.-H. Tseng \cite{TY19} where we will prove the relationship between all genera relative and orbifold invariants with stacky targets. Readers who do not want to assume Theorem \ref{thm:stacky-rel=orb}, may restrict their attention to the computation for toric Calabi--Yau manifolds instead of toric Calabi--Yau orbifolds in Section \ref{sec:compute-rel-inv}.

\subsection{Mirror theorem for relative pairs}

 In \cite{FTY}, a mirror theorem for relative pairs was stated in the language of Givental's Lagrangian cone which was recently developed in \cite{FWY} for relative Gromov--Witten theory. \cite[Theorem 1.4]{FTY} states the relative mirror theorem with non-extended $I$-functions which relates to relative Gromov--Witten invariants with one relative marking (hence, with maximal tangency condition). \cite[Theorem 1.5]{FTY} states the relative mirror theorem with extended $I$-functions which relates to relative Gromov--Witten invariants with more than one relative markings with possibly negative contact orders. In this paper, we only need to consider relative Gromov--Witten invariants with one relative marking. Therefore, we will only state (a version of) \cite[Theorem 1.4]{FTY}.

\begin{defn}
Let $X$ be a smooth projective variety and $D$ be a smooth nef divisor, the (non-extended) $J$-function for the pair $(X,D)$ is defined as
\[
J_{(X,D)}(\tau,z)=1+\sum_{\substack{(d,l)\neq (0,0)\\ d\in H_2^{\on{eff}}(\mathcal X)}}\sum_{\alpha}\frac{q^{d}}{l!}\left\langle \frac{\phi_\alpha}{z(z-\bar{\psi})},\tau,\ldots, \tau\right\rangle_{0,1+l, d}^{(X,D)}\phi^{\alpha},
\]
where  $\tau\in H^*(X)$; $\{\phi_\alpha\}$ is a basis of the ambient cohomology ring $i^*(H^*(X))$ of $H^*(D)$, that is, the pullback of $H^*(X)$ via the inclusion map $i:D\hookrightarrow X$; $\{\phi^\alpha\}$ is the dual basis under the Poincar\'e pairing. 
\end{defn}

Note that, the markings with insertion $\tau\in H^*(X)$ are interior markings, in other words, not relative markings. Therefore, the only relative marking is the distinguish marking (the first marking). The only relative marking has to carry maximal contact order: $D\cdot d$. We still call the first marking a relative marking even when $D\cdot d=0$.

Then we consider the corresponding (non-extended) $I$-function.
\begin{defn}
The (non-extended) $I$-function of the smooth relative pair $(X,D)$ is 
\[
I_{(X,D)}(y,z)=\sum_{d\in H_2^{\on{eff}}(X)}J_{X,d}(\tau,z)y^d\left(\prod_{0<a\leq D\cdot d-1}(D+az)\right),
\]
which is considered as an $H^*(D)$-valued function under the restriction map
\[
i^*: H^*(X)\rightarrow H^*(D).
\]
\end{defn}

\begin{conv}\label{conv-1}
We use the convention that $\prod_{0<a\leq D\cdot d-1}(D+az)=1$ when $D\cdot d=0$. This is because the $I$-function for the pair $(X,D)$ is taken as the limit of the $I$-function for the root stack $X_{D,r}$ as $r\rightarrow \infty$. According to \cite[Theorem 1.1]{FTY}, the hypergeometric factor for the $I$-function of $X_{D,r}$ is
\[
\frac{\prod_{0<a\leq D\cdot d}(D+az)}{\prod_{\langle a\rangle=\langle (D\cdot d)/r\rangle,0<a\leq (D\cdot d)/r}(\frac Dr+az)}=\frac{\prod_{a\leq D\cdot d}(D+az)}{\prod_{a\leq 0}(D+az)} \cdot\frac{\prod_{\langle a\rangle=\langle (D\cdot d)/r\rangle,a\leq 0}(\frac Dr+az)}{\prod_{\langle a\rangle=\langle (D\cdot d)/r\rangle,a\leq (D\cdot d)/r}(\frac Dr+az)}.
\]
Therefore, it is $1$ when $D\cdot d=0$.
\end{conv}

\begin{theorem}[\cite{FTY}, Theorem 1.4]\label{thm:rel-mirror}
Let $X$ be a smooth projective variety and $D$ be a smooth nef divisor such that $-K_X-D$ is nef. Then the $I$-function $I_{(X,D)}(y,z)$ coincides with the $J$-function $J_{(X,D)}(q,z)$ via change of variables, called mirror map.
\end{theorem}

\begin{remark}
In \cite{FTY}, the $I$-function and the $J$-function both take value in the infinite dimensional state space 
\[
\HH=\bigoplus\limits_{i\in\mathbb Z}\HH_i,
\]
where $\HH_0=H^*(X)$ and $\HH_i=H^*(D)$ if $i\in \mathbb Z \setminus \{0\}$. The non-extended $I$-function and the non-extended $J$-function actually take value in $\bigoplus_{i\in \mathbb Z_{<0}}\HH_i$, where each copy $\HH_i$ is $H^*(D)$. In Theorem \ref{thm:rel-mirror}, we simply identify each $\HH_i$ with the same $H^*(D)$ for $i\in \mathbb Z_{<0}$.
\end{remark}

\begin{remark}
The proof of Theorem \ref{thm:rel-mirror} requires orbifold quantum Lefschetz \cite{Tseng} and the mirror theorem for toric stack bundles \cite{JTY}. As mentioned in \cite[Remark 3.4]{FTY} that neither orbifold quantum Lefschetz nor the mirror theorem for toric stack bundles is true for general stacky pairs $(\mathcal X,\mathcal D)$. However, it will not be a problem for us, since we will restrict our attention to toric pairs. A root stack of a toric orbifold along a toric divisor is still a toric orbifold whose mirror theorem is already known in \cite{CCIT}. Therefore, Theorem \ref{thm:stacky-rel=orb} directly implies a mirror theorem for toric orbifold pairs. See Theorem \ref{thm:mirror-toric-orbi-pair}.
\end{remark}

\section{Preliminaries on toric orbifolds}\label{sec:toric}

\subsection{Construction}\label{sec:def-toric}

Following the construction of \cite{BCS}, a toric orbifold is defined by a stacky fan $(N,\Sigma,\rho)$, where 
\begin{itemize}
\item
$N$ is a lattice of rank $n$; 
\item
$\Sigma \subset N_{\mathbb{R}}=N\otimes_{\mathbb{Z}}\mathbb{R}$ 
is a rational simplicial fan;
\item
$\rho:\mathbb{Z}^{m}\rightarrow N$ is a map given by $\{b_0,\cdots, b_{m-1}\}\subset N$. The $b_i$'s are vectors determining the rays of the stacky fan.
\end{itemize}
We denote by $|\Sigma|\subset N_{\mathbb R}$ the support of $\Sigma$.

Following \cite{Jiang}, we can include some additional vectors $b_m,\ldots, b_{m^\prime-1}\in N\cap |\Sigma|$. The data
\[
(N,\Sigma,\rho^S)
\]
is called an extended stacky fan, where
\begin{align*}
\rho^S:\mathbb Z^{m^\prime}\rightarrow N&\\
e_i\mapsto b_i,& \text{ for }i\in\{0,1,\ldots, m^\prime -1\},
\end{align*}
where $\{e_i\}_{i=0}^{m^\prime-1}$ is the set of standard unit vectors of $\mathbb Z^{m^\prime}$.
In this paper, we choose $\{b_j\}_{j=m}^{m^\prime -1}$ such that $\{b_i\}_{i=0}^{m-1}\cup\{b_j\}_{j=m}^{m^\prime-1}$ generates $N$ over $\mathbb Z$.

The fan sequence is
\begin{equation}\label{fan-seq}
0 \longrightarrow \mathbb{L}:=\text{ker}(\rho^S) \stackrel{\psi^\vee}{\longrightarrow} \mathbb{Z}^{m^\prime} \stackrel{\rho^S}{\longrightarrow} N\longrightarrow 0.
\end{equation}

For $I\subset\{0,1,2,\cdots,m^\prime-1\}$, let $\sigma_I$ be the cone generated by 
$b_i, i\in I$ and let $\overline{I}$ be the complement of $I$ in $\{0,1,2,\cdots,m^\prime-1\}$.  The collection of anti-cones $\mathcal{A}$ is defined as follows: 
\[
\mathcal{A}:=\left\{I\subset \{0,1,2,\cdots, m^\prime-1\}: \sigma_{\overline{I}}\in \Sigma \right\}.
\]
For $I\subset \{0,1,...,m^\prime-1\}$, define \[
\mathbb{C}^{I}:=\left\{ (z_{0},\ldots,z_{m^\prime-1}):z_{i}=0
\text{ for } i \not\in I\right\}.
\]
 Let $\cU$ be the open subset of $\mathbb{C}^{m^\prime}$ defined as
\begin{equation*}
\mathcal{U}:=\mathbb{C}^{m^\prime}\setminus \cup_{I\not\in \mathcal{A}}\mathbb{C}^{I}. 
\end{equation*}

The algebraic torus $G:=\mathbb L\otimes_{\mathbb Z}\mathbb C^\times$ acts on $\mathbb C^{m^\prime}$ via the map
\[
G:=\mathbb L\otimes_{\mathbb Z}\mathbb C^\times \longrightarrow (\mathbb C^\times)^{m^\prime},
\]
which is obtained by tensoring $\mathbb C^\times$ to the fan sequence (\ref{fan-seq}).

\begin{defn}[see \cite{BCS}] 
The toric orbifold $\mathcal{X}(\Sigma)$ 
is defined as the quotient stack 
\[
\mathcal{X}(\Sigma):=[\mathcal{U}/G].
\]
\end{defn}

\begin{defn} [\cite{BCS}]
Given a cone $\sigma$ in a stacky fan $(N,\Sigma,\rho)$, 
we define the set of box elements $\bbox(\sigma)$ as follows
\[
\bbox(\sigma)=:
\left\{ v\in N: v=
\sum\limits_{b_k\subseteq  \sigma}c_{k}b_{k}
\text{ for some }c_{k}\in[0,1)\cap \mathbb Q \right\}.
\]
The set of box elements associated to the stacky fan $(N,\Sigma,\rho)$ is defined to be
\[
\bbox(\Sigma):
=\cup_{\sigma\in \Sigma}\bbox(\sigma).
\]
\end{defn}

According to \cite{BCS}, the connected components of the inertia stack $\mathcal{IX}(\Sigma)$ 
are indexed by the elements of  $\bbox(\Sigma)$. 
Given $v\in \bbox(\Sigma)$, we denote by $\mathcal X_v$ the corresponding connected component of 
$\mathcal{IX}$. When $v\neq 0$, the connected component $\mathcal X_v$ is called a twisted sector. The age of the twisted sector $\mathcal X_v$ is defined by  $\age(v):=\sum\limits_{b_k\subseteq \sigma} c_{k}$. We write $\bbox(\Sigma)^{\age=1}\subset \bbox(\Sigma)$ for the set of box elements with ages equal to $1$. 

The Chen--Ruan cohomology $H^*_{\on{CR}}(\mathcal X,\mathbb Q)$ for an orbifold is defined in \cite{CR04}. For the toric orbifold $\mathcal X(\Sigma)$, the Chen--Ruan cohomology is given by
\[
H^d_{\on{CR}}(\mathcal X(\Sigma,\mathbb Q))=\bigoplus_{v\in \bbox(\Sigma)}H^{d-2\age(v)}(\mathcal X_v,\mathbb Q).
\]
Applying $\Hom_{\mathbb Z}(-,\mathbb Z)$ to the fan sequence (\ref{fan-seq}) gives the divisor sequence:
\begin{align}\label{divisor-seq}
0\longrightarrow M\stackrel{(\rho^{S})^{\vee}}{\longrightarrow} (\mathbb{Z}^{m^\prime})^\vee\stackrel {\psi^{\vee}}{\longrightarrow}\mathbb{L}^{\vee}\rightarrow 0,
\end{align}
where $M:=N^\vee=\Hom(N,\mathbb Z)$ and $\mathbb L^\vee=\Hom(\mathbb L,\mathbb Z)$. The Picard group $\Pic(\mathcal{X}(\Sigma))$ 
of $\mathcal{X}(\Sigma)$
can be identified with the character group 
$\Hom(G,\mathbb{C}^{\times})$. 

Let $\{e_i^\vee\}_{i=0}^{m^\prime-1}\subset (\mathbb Z^{m^\prime})^\vee$ be the dual basis  of $\{e_i\}_{i=0}^{m^\prime-1}\subset (\mathbb Z^{m^\prime})$. We also set:
\[
D_i:=\psi^\vee(e_i^{\vee})\in \mathbb L^\vee, i\in\{0,\ldots, m^\prime-1\}.
\]
The collection $\{D_i\}_{i=0}^{m-1}$ of toric prime divisors corresponds to the generators  $\{b_i\}_{i=0}^{m-1}$. Moreover, we have
\[
H^{2}(\mathcal{X}(\Sigma);\mathbb{Z})\cong (\mathbb L^\vee\otimes \mathbb Q)/(\sum_{i=m}^{m^\prime-1}\mathbb Q D_i).
\]
There is a canonical splitting of the quotient map $\mathbb L^\vee \otimes \mathbb Q\rightarrow H^2(\mathcal X,\mathbb Q)$ as described in \cite[Section 3.1.2]{Iritani}. We briefly describe the splitting in here. For $m\leq j \leq m^\prime-1$, each $b_j$ is contained in a cone of $\Sigma$.  Let $I_j\in \mathcal A$ be the anticone of the minimal cone containing $b_j$. Then we have
\[
b_j=\sum_{i\not\in I_j}c_{ji}b_i\in N\otimes \mathbb Q,
\]
where $c_{ji}\in \mathbb Q_{\geq 0}$. By the divisor sequence (\ref{divisor-seq}) tensored with $\mathbb Q$, we can find unique $D_j^\vee\in \mathbb L\otimes \mathbb Q$ such that 
\begin{align}\label{D-dual}
D_i\cdot D_j^\vee=\left\{
\begin{array}{cc}
1& \text{ if } i=j,\\
-c_{ji}& \text{ if }i\not\in I_j,\\
0 & \text{ if }\in I_j\setminus \{j\}.
\end{array}
\right.
\end{align}
These vectors $D_j^\vee$ define a decomposition 
\[
\mathbb L^\vee \otimes \mathbb Q=\on{Ker}\left((D_m^\vee,\ldots, D_{m^\prime-1}^\vee):\mathbb L^\vee\otimes \mathbb Q\rightarrow \mathbb Q^{m^\prime-m}\right)\oplus \bigoplus_{j=m}^{m^\prime-1}\mathbb Q D_j.
\]
The factor $\on{Ker}\left((D_m^\vee,\ldots, D_{m^\prime-1}^\vee):\mathbb L^\vee\otimes \mathbb Q\rightarrow \mathbb Q^{m^\prime-m}\right)$ is identified with $H^2(\mathcal X,\mathbb Q)$ under the quotient map $\mathbb L^\vee \otimes \mathbb Q\rightarrow H^2(\mathcal X,\mathbb Q)$. Therefore, $H^2(\mathcal X,\mathbb Q)$ can be regarded as a subspace of $\mathbb L^\vee \otimes \mathbb Q$.

The extended K\"ahler cone $\tilde{C}_{\mathcal X}$ is defined as
\[
\tilde{C}_{\mathcal X}:=\bigcap_{I\in \mathcal A}(\sum_{i\in I}\mathbb R_{>0}D_i)\subset \mathbb L^\vee\otimes \mathbb R.
\]
The genuine K\"ahler cone $C_{\mathcal X}$ of $\mathcal X$ is the image of $\tilde {C}_{\mathcal X}$ under $\mathbb L^\vee \otimes \mathbb R\rightarrow H^2(\mathcal X,\mathbb R)$.

We set $r:=m^\prime -n$ and $r^\prime:=m-n$. Then the rank of $\mathbb L^\vee$ is $r$ and the rank of $H_2(\mathcal X,\mathbb Z)$ is $r^\prime$. We choose an integral basis $\{p_1,\ldots, p_r\}$ of $\mathbb L^\vee$ such that $p_a$ is in the closure $\on{cl}(\tilde{C}_{\mathcal X})$ of $\tilde{C}_{\mathcal X}$ for all $a\in\{1,\ldots,r\}$ and $p_{r^\prime+1},\ldots,p_r$ are in $\sum_{i=m^\prime+1}^m\mathbb R_{\geq 0}D_i$. The images $\bar{p}_1,\ldots,\bar{p}_{r^\prime}$ of $p_1,\ldots,p_{r^\prime}$ in $H^2(\mathcal X,\mathbb R)$ are nef and the images $\bar{p}_{r^\prime+1},\ldots,\bar{p}_r$ of $p_{r^\prime+1},\ldots,p_r$ are zero. We define a matrix $(m_{ia})$ by
\[
D_i=\sum_{a=1}^r m_{ia}p_a, \quad m_{ia}\in \mathbb Z.
\]
Let $\bar {D}_i$ be the image of $D_i$ under the quotient map $\mathbb L^\vee\otimes \mathbb Q\rightarrow H^2(\mathcal X,\mathbb Q)$. Then for $i=0,\ldots,m-1$, the toric divisor $\bar{D}_i$ is given by
\[
\bar{D}_i=\sum_{a=1}^{r^\prime}m_{ia}\bar{p}_{a}.
\]
For $m\leq j\leq m^\prime-1$, we have
\[
\bar{D}_j=0.
\]
Let $\{\gamma_1,\ldots, \gamma_r\}\subset \mathbb L$ be the dual basis of $\{p_1,\ldots,p_r\}\subset \mathbb L^\vee$ defined by
\[
\gamma_a=\sum_{i=0}^{m^\prime-1}m_{ia}e_i\in (\mathbb Z)^{m^\prime}.
\] 
Then $\{\gamma_1,\ldots,\gamma_{r^\prime}\}$ forms a basis of $H_2^{\on{eff}}(\mathcal X,\mathbb Q)$.

\begin{defn}[\cite{CCLT}, Definition 2.8]
A toric orbifold $\mathcal X$ is semi-Fano if $\hat{\rho}(\mathcal X):=\sum_{i=0}^{m^\prime-1}D_i$ is contained in the closure of the extended K\"ahler cone $\tilde{C}_{\mathcal X}$. 
\end{defn}

\begin{remark}
The semi-Fano condition depends on the choice of the extra vectors $b_m,\ldots, b_{m^\prime-1}$. By \cite[Lemma 3.3]{Iritani}, we have
\[
\hat{\rho}(\mathcal X)=c_1(\mathcal X)+\sum_{j=m}^{m^\prime-1}(1-\age(b_j))D_j.
\] 
Hence, the semi-Fano condition holds if and only if the first Chern class $c_1(\mathcal X)\in H^2(\mathcal X;\mathbb Q)$ of $\mathcal X$ is contained in the closure of the K\"ahler cone $C_{\mathcal X}$ and $\age(b_j)\leq 1$ for $m\leq j\leq m^\prime-1$.
\end{remark}

When $\mathcal X$ is a toric variety, the semi-Fano condition is equivalent to the condition that the anticanonical divisor $-K_{\mathcal X}$ is nef.

Following \cite{CCLT}, we only consider toric orbifolds that satisfy the following assumption.

\begin{assumption}\label{assumption-extra-vectors}
We assume that we always choose the extra vectors $b_m\ldots,b_{m^\prime-1}$ such that
\[
\{b_m\ldots, b_{m^\prime-1}\} =\bbox(\Sigma)^{\age=1}:=\{v\in \bbox(\Sigma)|\age(v)=1\}
\]
and $\{b_0,\ldots, b_{m^\prime-1}\}$ generates $N$ over $\mathbb Z$.
\end{assumption}

As in \cite{Iritani}, we consider the following two subsets of $\mathbb L\otimes \mathbb Q$:
\[
\mathbb K=\{d\in \mathbb L\otimes \mathbb Q;\{i\in\{0,\ldots,m^\prime-1\};D_i\cdot d\in \mathbb Z\}\in \mathcal A\};
\]
\[
\mathbb K_{\on{eff}}=\{d\in \mathbb L\otimes \mathbb Q;\{i\in\{0,\ldots,m^\prime-1\};D_i\cdot d\in \mathbb Z_{\geq 0}\}\in \mathcal A\}.
\]
For a real number $r$, let $\lceil r\rceil$, $\lfloor r\rfloor$ and $\langle r\rangle$ denote the ceiling, floor and fractional part of $r$ respectively. For $d\in \mathbb K$, we define
\[
v(d):=\sum_{i=0}^{m^\prime-1} \lceil D_i\cdot d\rceil b_i \in N.
\]
Note that $v(d)$ is a box element since
\[
v(d)=\sum_{i=0}^{m^\prime-1}(\langle -D_i\cdot d\rangle+D_i\cdot d)b_i=\sum_{i=0}^{m^\prime-1}\langle-D_i\cdot d \rangle b_i \in N\otimes \mathbb Q,
\]
by the fan sequence (\ref{fan-seq}). Therefore, if $v(d)$ is not zero, it corresponds to a twisted sector $\mathcal X_{v(d)}$ of the inertia stack $\mathcal {IX}$ of $\mathcal X$.
\subsection{Toric mirror theorems}

In this section, we state the mirror theorem for toric orbifolds proved in \cite{CCIT}. Recently, a mirror theorem for smooth relative pairs is proved in \cite{FTY}. In this paper, we will need a version of the mirror theorem in \cite{FTY} for toric orbifold pairs $(\mathcal X,D)$, that is, a toric orbifold $\mathcal X$ along with its toric prime divisor $D$. However, the mirror theorem in \cite{FTY} is only proved for the case when $\mathcal X$ is a smooth projective variety. In the forthcoming paper \cite{TY19}, we will extend the relation between relative and orbifold Gromov--Witten invariants to the pair $(\mathcal X,D)$ when $\mathcal X$ is a smooth Deligne-Mumford stack. However, such relation does not imply a mirror theorem for orbifold relative pairs in general, since the corresponding mirror theorem for root stacks is not known to be true as explained in \cite[Remark 3.4]{FTY}. In this paper, we only consider orbifold toric pairs, whose mirror theorem holds since the corresponding root stacks are still toric orbifolds and a mirror theorem for toric orbifolds is proved in \cite{CCIT}.

We first consider the $I$-function for a toric orbifold $\mathcal X$.

\begin{defn}
The $I$-function of a toric orbifold $\mathcal X$ is an $H^*_{\on{CR}}(\mathcal X)$-valued power series defined by
\[
I_{\mathcal X}(y,z)=e^{t/z}\sum_{d\in\mathbb K_{\on{eff}}}y^{d}\left(\prod_{i=0}^{m^\prime-1}\frac{\prod_{a\leq 0,\langle a\rangle=\langle D_i\cdot d\rangle}(\bar{D}_i+az)}{\prod_{a \leq D_i\cdot d,,\langle a\rangle=\langle D_i\cdot d\rangle}(\bar{D}_i+az)}\right)\textbf{1}_{v(d)},
\]
where $t=\sum_{a=1}^r \bar p_a \log y_a$, $y^d=y_1^{p_1\cdot d}\cdots y_r^{p_r\cdot d}$ and, $\textbf{1}_{v(d)}\in H^0(\mathcal X_{v(d)})$ is the fundamental class of the twisted sector $\mathcal X_{v(d)}$.
\end{defn}

By \cite[Lemma 4.2]{Iritani}, when $\mathcal X$ is semi-Fano, the $I$-function is a convergent power series in $y_1,\ldots,y_r$. The $I$-function can be expanded as follows
\[
I_{\mathcal X}(y,z)=1+\frac{\tau(y)}{z}+O(z^{-2}),
\]
where $\tau(y)\in H_{\on{CR}}^{\leq 2}(\mathcal X)$ is called the mirror map.

Then we consider the $J$-function for a toric orbifold $\mathcal X$.
\begin{defn}
The $J$-function of a toric orbifold $\mathcal X$ is defined as
\[
J_{\mathcal X}(\tau,z)=e^{\tau_{0,2}/z}\left(1+\sum_{\substack{(d,l)\neq (0,0)\\ d\in H_2^{\on{eff}}(\mathcal X)}}\sum_{\alpha}\frac{q^{d}}{l!}\left\langle \frac{\phi_\alpha}{z(z-\bar{\psi})},\tau_{tw},\ldots, \tau_{tw}\right\rangle_{0,1+l, d}^{\mathcal X}\phi^{\alpha}\right),
\]
where $\tau_{0,2}=\sum_{a=1}^r \bar p_a\log q_a\in H^2(\mathcal X)$, $\tau_{tw}=\sum_{j=m}^{m^\prime-1}\tau_{b_j}\textbf{1}_{b_j}$, $q^d=q_1^{p_1\cdot d}\cdots q_{r^\prime}^{p_{r^\prime}\cdot d}$ and, $\{\phi_\alpha\},\{\phi^\alpha\}$ are dual basis of $H^*_{\on{CR}}(\mathcal X)$.
\end{defn}

The mirror theorem for a toric orbifold $\mathcal X$ can be stated as an equality between the $J$-function and the $I$-function via the mirror map.
\begin{theorem}[\cite{CCFK}, Theorem 3.12; \cite{CCIT}, Theorem 31]\label{thm:toric-stack}
Let $\mathcal X$ be a smooth semi-Fano toric orbifold with a projective coarse moduli space, then we have
\[
J_{\mathcal X}(\tau(y),z)=I_{\mathcal X}(y,z),
\]
where $\tau(y)$ is the mirror map.
\end{theorem}

For semi-projective semi-Fano toric orbifolds (e.g. toric Calabi--Yau orbifolds), one can also state the equivariant mirror theorem.  The mirror map is still denoted by $\tau(q)$.

Now we turn our attention to the mirror theorem for a toric pair $(\mathcal X,D)$, where $\mathcal X$ is a toric orbifold and $D$ is a toric divisor of $\mathcal X$. We consider the (non-extended) $J$-function and (non-extended) $I$-function for the pair $(\mathcal X,D)$.
\begin{defn}
The (non-extended) $J$-function for a toric pair $(\mathcal X,D)$ is defined as
\[
J_{(\mathcal X,D)}(\tau,z):=e^{\tau_{0,2}/z}\left(1+\sum_{\substack{(l, d)\neq (0,0)\\d\in H_2^{\on{eff}}(\mathcal X)}}\sum_{\alpha}\frac{q^{d}}{l!}\left\langle \frac{\phi_\alpha}{z(z-\bar{\psi})},\tau_{tw},\ldots, \tau_{tw}\right\rangle_{0,1+l, d}^{(\mathcal X,D)}\phi^{\alpha}\right),
\]
where $\{\phi_\alpha\}$ is a basis of the ambient cohomology ring $i^*(H^*_{\on{CR}}(\mathcal X))$ of $H^*_{\on{CR}}(D)$, that is, the pullback of $H^*_{\on{CR}}(\mathcal X)$ via the inclusion map $i:D\hookrightarrow \mathcal X$; $\{\phi^\alpha\}$ is the dual basis under the Poincar\'e pairing. 
\end{defn}
\begin{defn}
The (non-extended) $I$-function for a toric pair $(\mathcal X,D)$ is defined as
\[
I_{(\mathcal X,D)}(y,z)
= e^{t/z}\sum_{d\in \mathbb K_{\on{eff}}}y^{d}\left(\prod_{i=0}^{m^\prime-1}\frac{\prod_{a\leq 0,\langle a\rangle=\langle D_i\cdot d\rangle}(\bar{D}_i+az)}{\prod_{a \leq D_i\cdot d,,\langle a\rangle=\langle D_i\cdot d\rangle}(\bar{D}_i+az)}\right)\frac{\prod_{a\leq D\cdot d-1}(\bar{D}+az)}{\prod_{a \leq 0}(\bar{D}+az)}\textbf{1}_{v(d)}\\
\]
\end{defn}

By Theorem \ref{thm:stacky-rel=orb}, the genus zero relative Gromov--Witten invariants of $(\mathcal X,D)$ are equal to the genus zero orbifold Gromov--Witten invariants of the $r$-th root stack $\mathcal X_{D,r}$ of $\mathcal X$ along the toric divisor $D$ for sufficiently large $r$. Therefore, Theorem \ref{thm:stacky-rel=orb} and Theorem \ref{thm:toric-stack} together imply the following theorem.
\begin{thm}\label{thm:mirror-toric-orbi-pair}
Let $\mathcal X$ be a smooth semi-Fano toric orbifold with a projective coarse moduli space and $D$ be a toric prime divisor such that $-K_{\mathcal X}-D$ is nef, then we have
\[
J_{(\mathcal X,D)}(\tau(y),z)=I_{(\mathcal X,D)}(y,z),
\]
where $\tau(y)$ is called the (relative) mirror map which is the coefficient of the $z^{-1}$-term in $I_{(\mathcal X,D)}(y,z)$.
\end{thm}

For simplicity, readers can also restrict their attention to the case when $\mathcal X$ is a toric variety. Then Theorem \ref{thm:mirror-toric-orbi-pair} is just a special case of Theorem \ref{thm:rel-mirror} for toric varieties and their toric prime divisors.

\subsection{Open Gromov--Witten invariants of toric orbifolds}

Another object that we want to review is the genus $0$ open orbifold Gromov--Witten invariants of toric orbifolds which is defined in \cite{CP}. We refer readers to \cite{CP} and \cite[Section 3]{CCLT} for more details about the definition.

We briefly recall the definition of Chern--Weil Maslov index following \cite{CS}. See also \cite[Appendix A]{CCLT}. Let $\mathcal C$ be a bordered orbifold Riemann surface with interior orbifold marked points $z_1^+,\ldots, z_l^+\in \mathcal C$ such that the orbifold structure at each $z_j^+$ is given by a branched covering map $z\mapsto z^{m_j}$ for some positive integer $m_j$. Let $\mathcal E$ be an orbifold vector bundle over $\mathcal C$ and $\mathcal L$ be a Lagrangian subbundle over the boundary $\partial \mathcal C$. According to \cite[Definition 2.3]{CS}, a unitary connection $\nabla$ of $\mathcal E$ is called $\mathcal L$-orthogonal if $\mathcal L$ is preserved by the parallel transport via $\nabla$ along the boundary $\partial \mathcal C$.

\begin{definition}[\cite{CS}, Definition 6.4]
The Chern--Weil Maslov index of the bundle pair $(\mathcal E, \mathcal L)$ is defined by
\[
\mu_{CW}(\mathcal E,\mathcal L)=\frac{\sqrt{-1}}{\pi}\int_{\mathcal C} \on{tr}(F_{\nabla}),
\]
where $F_{\nabla}\in \Omega^2(\mathcal C, \on{End}(\mathcal E))$ is the curvature induced by an $\mathcal L$-orthogonal connection $\nabla$.
\end{definition}

By \cite[Proposition 6.10]{CS}, the desingularized Maslov index $\mu^{\on{de}}$ defined in \cite[Section 3]{CP} is related to the Chern--Weil Maslov indices as follows:
\[
\mu_{CW}(\mathcal E,\mathcal L)=\mu^{\on{de}}(\mathcal E,\mathcal L)+2\sum_{j=1}^l\on{age}(\mathcal E;z_j^+).
\]
Let $w: (\mathcal C,\partial \mathcal C)\rightarrow (\mathcal X,L)$ be a holomorphic map from a bordered orbifold Riemann surface $\mathcal C$ to a K\"ahler orbifold $\mathcal X$ of complex dimension $n$ such that $w(\partial \mathcal C)$ is contained in the Lagrangian submanifold $L$. If $\beta\in \pi_2(\mathcal X,L)$ is represented by a holomorphic map $w$, then we put
\[
\mu_{CW}(\beta):=\mu_{CW}(w):=\mu_{CW}(w^*T\mathcal X,w^*TL).
\]
Suppose $\mathcal X$ is equipped with a non-zero meromorphic $n$-form $\Omega$ which has at worst simple poles. Let $D\subset \mathcal X$ be the pole divisor of $\Omega$ such that the generic points of $D$ are smooth. By \cite[Lemma A.3]{CCLT}, for a special Lagrangian submanifold $L\subset \mathcal X$, the Chern--Weil Maslov index of a class $\beta\in \pi_2(\mathcal X,L)$ is given by
\[
\mu_{CW}(\beta)=2\beta\cdot D.
\]
The values of Chern--Weil Maslov indices are computed in \cite[Theorem 6.2]{CP} for toric orbifolds. We also refer to \cite[Section 3.1]{CCLT} for more details for toric Calabi--Yau orbifolds.

Let $\mathcal X$ be a toric orbifold defined in Section \ref{sec:def-toric}. Let $L\subset \mathcal X$ be a Lagrangian torus fiber of the moment map
\[
\mu:\mathcal X\rightarrow M_{\mathbb R}:=M\otimes _{\mathbb Z}\mathbb R.
\] 
Given $i\in\{0,1,\ldots,m-1\}$, there is a smooth holomorphic disk intersecting the associated toric prime divisor $D_i$ with multiplicity one. Its homotopy class is denoted by $\beta_i\in \pi_2(\mathcal X, L)$. The Chern--Weil Maslov index is 
\[
\mu_{CW}(\beta_i)=2.
\] 
Given a twisted sector $v\in \bbox(\Sigma)$ of the toric orbifold $\mathcal X$, there is a unique holomorphic orbi-disk whose homotopy class is denoted by $\beta_v \in \pi_2(\mathcal X, L)$ with Chern--Weil Maslov index 
\[
\mu_{CW}(\beta_v)=2\on{age}(v).
\] 
By \cite[Lemma 9.1]{CP}, the relative homotopy group $\pi_2(\mathcal X, L)$ is generated by the classes $\beta_i$ for $i=0,\ldots,m-1$ together with $\beta_v$ for $v\in \bbox(\Sigma)\setminus\{0\}$. These generators of $\pi_2(\mathcal X,L)$ are called basic disk classes. Basic disk classes are the analogue of Maslov index two disk classes for toric manifolds. 
The following classification of basic disk classes is given in \cite[Corollary 6.3 and Corollary 6.4]{CP}:
\begin{itemize}
\item The smooth holomorphic disks of Maslov index two are in a one-to-one correspondence with the stacky vectors $\{b_0,\ldots,b_{m-1}\}$.
\item The holomorphic orbi-disks with one interior orbifold marked point and desingularized Maslov index zero are in a one-to-one correspondence with the twisted sectors $v\in \bbox(\Sigma)\setminus \{0\}$ of the toric orbifold $\mathcal X$.
\end{itemize}

Let $\beta\in \pi_2(\mathcal X,L)=H_2(\mathcal X,L;\mathbb Z)$ be a relative homotopy class. We consider holomorphic orbi-disks in $\mathcal X$ bounded by $L$ and representing the class $\beta$. Let $\mathcal X_{v_1},\ldots, \mathcal X_{v_l}$ be twisted sectors of $\mathcal X$. Consider the moduli space $\mathcal M_{1,l}^{main}(L,\beta, \vec x)$ of good representable stable maps from bordered orbifold Riemann surfaces of genus zero with one boundary marked point and $l$ interior marked points of type $\vec x=(\mathcal X_{v_1},\ldots,\mathcal X_{v_l})$ representing the class $\beta\in \pi_2(\mathcal X,L)$. The superscript "$main$" means that we have chosen a connected component on which the boundary marked points respect the cyclic order of $S^1=\partial D^2$. Following \cite{CP}, the moduli space $\mathcal M_{1,l}^{main}(L,\beta,\vec x)$ carries a non-vanishing virtual fundamental chain only if the following equality for the Chern--Weil Maslov index $\mu_{CW}(\beta)$ of $\beta$ holds:
\[
\mu_{CW}(\beta)=2+\sum_{j=1}^l(2\age(v_j)-2).
\]
The following compactness of the disk moduli is proved in \cite[Proposition 6.10]{CCLT}
\begin{proposition}[\cite{CCLT}, Proposition 6.10(c)]
The disk moduli $\mathcal M_{1,l}^{main}(L,\beta_i+\alpha,\vec x)$ for every $i=0,\ldots, m^\prime-1$ and $\alpha\in H_2(\mathcal X,\mathbb Z)$ is compact. 
\end{proposition}

\begin{defn}
Let $\beta=\beta_i+\alpha\in \pi_2(\mathcal X,L)$ be a relative homotopy class,  for $i\in\{0,\ldots, m^\prime-1\}$ and $\alpha\in H_2(\mathcal X,\mathbb Z)$. The open orbifold Gromov--Witten invariant $n_{1,l,\beta}^{\mathcal X}([\on{pt}]_L;{\textbf 1}_{v_1},\ldots,{\textbf 1}_{v_l})\in \mathbb Q$ is defined to be the push-forward
\[
n_{1,l,\beta}^{\mathcal X}([\on{pt}]_L;{\textbf 1}_{v_1},\ldots,{\textbf 1}_{v_l}):=\on{ev}_{0*}([\mathcal M_{1,l}(L,\beta,\vec x)]^{vir})\in H_n(L,\mathbb Q)\cong \mathbb Q,
\]
where $\on{ev}_0:\mathcal M_{1,l}^{main}(L,\beta_i+\alpha,\vec x)\rightarrow L$ is evaluation at the boundary marked point; $[\on{pt}]_L\in H^n(L,\mathbb Q)$ is the point class of $L$; and ${\textbf 1}_{v_i}\in H^0(\mathcal X_{v_i},\mathbb Q)$ is the fundamental class of the twisted sector $\mathcal X_{v_i}$.
\end{defn}

One can consider the following generating functions of open Gromov--Witten invariants
\[
1+\delta_i^{\on{open}}:=\sum_{\alpha\in H_2^{\on{eff}}(\mathcal X)}\sum_{l\geq 0}\sum_{v_1,\ldots v_l\in \bbox(\Sigma)^{\age=1}}\frac{\prod_{i=1}^l \tau_{v_i}}{l!}n_{1,l,\beta_{i}+\alpha}^{\mathcal X}( [\on{pt}]_L,\textbf{1}_{v_1},\ldots,\textbf{1}_{v_l})q^{\alpha},
\]
where $\beta_{i}$ is a basic smooth disk class for some $i\in\{0,1,\ldots,m-1\}$;

\[
\tau_v+\delta_v^{\on{open}}:=\sum_{\alpha\in H_2^{\on{eff}}(\mathcal X)}\sum_{l\geq 0}\sum_{v_1,\ldots v_l\in \bbox(\Sigma)^{\age=1}}\frac{\prod_{i=1}^l \tau_{v_i}}{l!}n_{1,l,\beta_{v}+\alpha}^{\mathcal X}( [\on{pt}]_L,\textbf{1}_{v_1},\ldots,\textbf{1}_{v_l})q^{\alpha},
\]
where $\beta_v$ is a basic orbi-disk class corresponding to a box element $v\in \bbox(\Sigma)^{\age=1}$.

\subsection{Toric compactifications of toric Calabi--Yau orbifolds}\label{sec:compactification}

\begin{definition}[\cite{CCLT}, Definition 4.1]\label{defn-toric-CY}
A Gorenstein toric orbifold $\mathcal X$ is called Calabi--Yau if there exists a dual vector $\underline{v}\in M=N^\vee$ such that $(\underline{v},b_i)=1$ for all stacky vectors $b_i$.
\end{definition}

Let $\mathcal X$ be a toric Calabi--Yau orbifold associated to a stacky fan $(N, \Sigma,b_0,\ldots, b_{m-1})$. Then $\mathcal X$ is always simplicial. Following \cite[Setting 4.3]{CCLT}, we assume that the coarse moduli space of the toric Calabi--Yau orbifold $\mathcal X$ is semi-projective.

Let $\beta^\prime\in \pi_2(\mathcal X,L)$ be a basic disk class with Chern--Weil Maslov index $2$, where $L$ is a Lagrangian torus fiber of the moment map $\mu_0:\mathcal X\rightarrow M\otimes_{\mathbb Z}\mathbb R$. We have $\partial \beta^\prime=b_{i_0}\in N$ for some $i_0\in\{0,1,\ldots, m^\prime-1\}$.
The toric compactification $\bar{\mathcal X}$ of $\mathcal X$ is constructed in \cite[Construction 6.1]{CCLT} by adding a vector 
\[
b_{\infty}:=-b_{i_0}\in N.
\]
The stacky fan for the toric compactification is 
\[
(N, \bar{\Sigma},\{b_i\}_{i=0}^{m-1}\cup \{b_\infty\}),
\]
where $\bar{\Sigma}\subset N_{\mathbb R}$ is the smallest complete simplicial fan that contains $\Sigma$ and the ray $\mathbb R_{\geq 0}b_\infty\subset N_{\mathbb R}$. 
\begin{definition}
The toric compactification $\bar{\mathcal X}$ is defined as the toric orbifold associated to $\bar{\Sigma}$. That is, 
\[
\bar{\mathcal X}:=\mathcal X({\bar{\Sigma}}).
\]
The extra vectors $\{b_m,\ldots,b_{m^\prime-1}\}\subset N$ are chosen to be the same as that for $\mathcal X$.
\end{definition}
We further assume that $\bar{\mathcal X}$ is projective which means that $b_{i_0}\in N$ lies in the interior of the support $|\Sigma|$, see \cite[Proposition 6.5]{CCLT}.
\begin{remark}
The definition of the toric compatification $\bar{\mathcal X}$ of $\mathcal X$ depends on the class $\beta^\prime$.
\end{remark}
We have $\mathcal X\subset \bar{\mathcal X}$. We write $D_\infty:=\bar{\mathcal X}\setminus \mathcal X$ for the toric prime divisor corresponding to $b_\infty$.

Let $\beta_\infty\in \pi_2(\bar{\mathcal X},L)$ be the basic disk class corresponding to $b_\infty$. The class $\bar{\beta}^\prime:=\beta^\prime+\beta_\infty$ belongs to $H_2(\bar{\mathcal X},\mathbb Q)$ since $\partial (\beta^\prime+\beta_\infty)=b_0+b_\infty=0\in N$. We have $c_1(\bar{\mathcal X})\cdot \bar{\beta}^\prime=2$ and $D_\infty\cdot\bar{\beta}^\prime=1$. Moreover, we have the decompositions
\[
H_2(\bar{\mathcal X},\mathbb Q)=H_2(\mathcal X,\mathbb Q)\oplus \mathbb Q\bar{\beta}^\prime \text{ and } H_2^{\on{eff}}(\bar{\mathcal X})=H_2^{\on{eff}}(\mathcal X)\oplus \mathbb Z_{\geq 0}\bar{\beta}^\prime.
\]
For any $\alpha\in H_2^{\on{eff}}(\mathcal X)$, we have $D_\infty\cdot \alpha=0$ and $c_1(\bar{\mathcal X})\cdot \alpha=0$.
We write $\bar{\mathbb L},\bar{\mathbb K}$ and $\bar{\mathbb K}_{\on{eff}}$ for the counterparts for $\bar{\mathcal X}$ of the spaces $\mathbb L, \mathbb K$ and $\mathbb K_{\on{eff}}$. Then we have the following decompositions
\[
\bar{\mathbb L}=\mathbb L\oplus \mathbb Z d_\infty, \quad \bar{\mathbb K}=\mathbb K\oplus \mathbb Z d_\infty, \quad \bar{\mathbb K}_{\on{eff}}=\mathbb K_{\on{eff}}\oplus \mathbb Z_{\geq 0}d_\infty,
\]
where $d_\infty=e_{i_0}+e_\infty\in \bigoplus_{i=0}^{m^\prime-1}\mathbb Z e_i\oplus \mathbb Z e_\infty$.

Consider $\bar{\mathbb L}^\vee:=\Hom(\mathbb L,\mathbb Z)$, note that 
\[
\on{rank}(\bar{\mathbb L}^\vee)=r+1=m^\prime-n+1, \quad \on{rank}(H^2(\bar{\mathcal X}))=r^\prime+1=m-n+1.
\]
We choose an integral basis
\[
\{p_1,\ldots,p_r,p_\infty\}\subset \bar{\mathbb L}^\vee
\]
such that $p_a$ is in the closure of $\tilde {C}_{\bar{\mathcal X}}$ for all $a$. We further require that $\{p_{r^\prime+1},\ldots,p_r\}\subset \sum_{i=m}^{m^\prime-1}\mathbb R_{\geq 0}D_i$ so that the images $\{\bar{p}_1,\ldots \bar{p}_{r^\prime},\bar{p}_\infty\}$ of $\{p_1,\ldots,p_{r^\prime},p_\infty\}$ under the quotient map $\bar{\mathbb L}^\vee\otimes\mathbb Q\rightarrow H^2(\bar{\mathcal X},\mathbb Q)$ form a nef basis of $H^2(\bar{\mathcal X},Q)$ and $\bar{p}_a=0$ for $a\in\{r^\prime+1,\ldots,r\}$.

We define a matrix $(m_{ia})$ by
\[
D_i=\sum_{a\in\{1,\ldots,r\}\cup\{\infty\}} m_{ia}p_a, \quad m_{ia}\in \mathbb Z.
\]
Let $\bar {D}_i$ be the image of $D_i$ under $\bar{\mathbb L}^\vee\otimes \mathbb Q\rightarrow H^2(\bar{\mathcal X},\mathbb Q)$. Then for $i\in\{0,\ldots,m-1,\infty\}$, the toric divisor $\bar{D}_i$ is given by
\[
\bar{D}_i=\sum_{a\in\{1,\ldots,r^\prime\}\cup\{\infty\}}m_{ia}\bar{p}_{a}.
\]
For $m\leq j\leq m^\prime-1$, we have
\[
\bar{D}_j=0.
\]

By \cite[Proposition 6.3]{CCLT}, the toric orbifold $\bar{\mathcal X}$ is semi-Fano.

\section{Computation of relative invariants}\label{sec:compute-rel-inv}

The goal of this section is to compute the following generating functions of relative invariants
\begin{align}\label{rel-gen-fun}
1+\delta_i^{\on{rel}}:=\sum_{\alpha\in H_2^{\on{eff}}(\mathcal X)}\sum_{l\geq 0}\sum_{v_1,\ldots v_l\in \bbox(\Sigma)^{\age=1}}\frac{\prod_{i=1}^l \tau_{v_i}}{l!}\left\langle [\on{pt}]_{D_\infty},\textbf{1}_{\bar{v}_1},\ldots,\textbf{1}_{\bar{v}_l}\right\rangle_{0,1+l, \bar{\beta}^\prime+\alpha}^{(\bar{\mathcal X},D_\infty)}q^{\alpha},
\end{align}
where $\bar{\beta}^\prime:=\beta^\prime+\beta_\infty$, $\beta^\prime=\beta_{i}$ is a basic smooth disk class for some $i\in\{0,1,\ldots,m-1\}$, and the toric compactification $\bar{\mathcal X}$ is constructed using $\beta_i$;
\begin{align}\label{rel-gen-fun-twisted}
\tau_v+\delta_v^{\on{rel}}:=\sum_{\alpha\in H_2^{\on{eff}}(\mathcal X)}\sum_{l\geq 0}\sum_{v_1,\ldots v_l\in \bbox(\Sigma)^{\age=1}}\frac{\prod_{i=1}^l \tau_{v_i}}{l!}\left\langle [\on{pt}]_{D_\infty},\textbf{1}_{\bar{v}_1},\ldots,\textbf{1}_{\bar{v}_l}\right\rangle_{0,1+l, \bar{\beta}^\prime+\alpha}^{(\bar{\mathcal X},D_\infty)}q^{\alpha},
\end{align}
where $\bar{\beta}^\prime:=\beta^\prime+\beta_\infty$, $\beta^\prime=\beta_{v}$ is a basic orbi-disk class corresponding to a box element $v\in \bbox(\Sigma)^{\age=1}$, and the toric compactification $\bar{\mathcal X}$ is constructed using $\beta_v$.

\subsection{Relative invariants via mirror theorem}
Let $\bar{\mathcal X}$ be the toric compactification of $\mathcal X$ corresponding to a basic (orbi-)disk class $\beta_{i_0}$ for some $i_0\in\{0,\ldots, m^\prime-1\}$. We consider the relative Gromov--Witten theory of $(\bar{\mathcal X},D_\infty)$. The $J$-function of $(\bar{\mathcal X},D_\infty)$ can be written as
\[
J_{(\bar{\mathcal X},D_\infty)}(\tau,z):=e^{\tau_{0,2}/z}\left(1+\sum_{\substack{(l, d)\neq (0,0)\\d\in H_2^{\on{eff}}(\bar{\mathcal X})}}\sum_{\alpha}\frac{q^{d}}{l!}\left\langle \frac{\phi_\alpha}{z(z-\bar{\psi})},\tau_{tw},\ldots, \tau_{tw}\right\rangle_{0,1+l, d}^{(\bar{\mathcal X},D_\infty)}\phi^{\alpha}\right),
\]
where 
\[
\tau_{0,2}\in H^2(\bar{\mathcal X})\text{ and } \tau_{tw}=\sum_{v\in \on{Box}(\Sigma)^{\on{age}=1}}\tau_v\textbf{1}_{\bar{v}}. 
\]
Note that the only relative marking is the first marking, $\{\phi_\alpha\}$ and $\{\phi^\alpha\}$ are dual basis of $i^*(H^*(\bar{\mathcal X}))$, where $i:D_\infty\hookrightarrow \bar{\mathcal X}$ is the inclusion map.

We want to extract invariants of the form
\[
\left\langle [\on{pt}]_{D_\infty},\textbf{1}_{\bar{v}_1},\ldots,\textbf{1}_{\bar{v}_l}\right\rangle_{0,1+l, \bar{\beta}}^{(\bar{\mathcal X},D_\infty)},
\]
where $\bar{\beta}=\bar{\beta}^\prime+\alpha$ for some $\alpha\in H_2^{\on{eff}}(\mathcal X)$. Note that the first marking is the only relative marking and $[\on{pt}]_{D_\infty}\in H^{2n-2}(D_\infty)$ is the point class of the divisor $D_\infty$.

Expanding the $J$-function into a power series in $1/z$ yields
\begin{align*}
e^{\tau_{0,2}/z}\left(1+\sum_{n, d}\sum_{\alpha}\frac{q^{d}}{n!}\sum_{k\geq 0}\left\langle \phi_\alpha\bar{\psi}^k,\tau_{tw},\ldots,\tau_{tw}\right\rangle_{0,1+l, d}^{(\bar{\mathcal X},D_\infty)}\frac{\phi^{\alpha}}{z^{k+2}}\right).
\end{align*}
Since there is no descendant insertions, we only consider the terms with $k=0$. Therefore, the invariants that we need are in the coefficient of the $1/z^2$-term of $J_{(\bar{\mathcal X},D_\infty)}$ that takes values in $H^0(D_\infty)$. On the other hand, invariants in the coefficient of the $1/z^2$-term of $J_{(\bar{\mathcal X},D_\infty)}$ that takes values in $H^0(D_\infty)$ must satisfy the virtual dimension constraint
\begin{align}\label{equ:vir-dim}
(\dim_{\mathbb C}\bar{\mathcal X}-3)+c_1(\bar{\mathcal X})\cdot d-D_\infty\cdot d+l+1=\sum_{i=1}^l \age(v_i)+\dim_{\mathbb C}(D_\infty)
\end{align}
Since $\dim_{\mathbb C}\bar{\mathcal X}=\dim_{\mathbb C}(D_\infty)+1$ and we choose $v_i\in \bbox(\Sigma)^{\age=1}$, the virtual dimension constraint (\ref{equ:vir-dim}) becomes
\[
c_1(\bar{\mathcal X})\cdot d-D_\infty\cdot d=1.
\]
Recall that $d\in H_2^{\on{eff}}(\bar{\mathcal X})=H_2^{\on{eff}}(\mathcal X)\oplus \mathbb Z_{\geq 0}\bar{\beta}^\prime.$ Moreover, $c_1(\bar{\mathcal X})\cdot \alpha=0$ and $D_\infty\cdot \alpha=0$ for $\alpha\in H_2^{\on{eff}}(\mathcal X)$; $c_1(\bar{\mathcal X})\cdot \bar{\beta}^\prime=2$ and $D_\infty\cdot\bar{\beta}^\prime=1$. Therefore ,we must have $d=\bar{\beta}^\prime+\alpha$ for some $\alpha\in H_2^{\on{eff}}(\mathcal X)$.
So we conclude that the coefficient of the $1/z^2$-term of $J_{(\bar{\mathcal X},D_\infty)}$ that takes values in $H^0(D_\infty)$ consists exactly the form of invariants that we want to extract.

Then, we turn our attention to the $I$-function of $(\bar{\mathcal X},D_\infty)$.
\begin{align*}
&I_{(\bar{\mathcal X},D_\infty)}(y,z)\\
=& e^{t/z}\sum_{d\in \bar{\mathbb K}_{\on{eff}}}y^{d}\left(\prod_{i=0}^{m^\prime-1}\frac{\prod_{a\leq 0,\langle a\rangle=\langle D_i\cdot d\rangle}(\bar{D}_i+az)}{\prod_{a \leq D_i\cdot d,,\langle a\rangle=\langle D_i\cdot d\rangle}(\bar{D}_i+az)}\right)\frac{\prod_{a\leq 0}(\bar{D}_\infty+az)}{\prod_{a \leq D_\infty\cdot d}(\bar{D}_\infty+az)}\frac{\prod_{a\leq D_\infty\cdot d-1}(\bar{D}_\infty+az)}{\prod_{a \leq 0}(\bar{D}_\infty+az)}\textbf{1}_{v(d)}\\
=& e^{t/z}\sum_{d\in \bar{\mathbb K}_{\on{eff}}}y^{d}\left(\prod_{i=0}^{m^\prime-1}\frac{\prod_{a\leq 0,\langle a\rangle=\langle D_i\cdot d\rangle}(\bar{D}_i+az)}{\prod_{a \leq D_i\cdot d,\langle a\rangle=\langle D_i\cdot d\rangle}(\bar{D}_i+az)}\right)\frac{1}{\bar{D}_\infty+(D_\infty\cdot d)z}\textbf{1}_{v(d)}\\
=&1+\frac{\tau(y)}{z}+h.o.t,
\end{align*}
where $\tau(y)$ is a function with values in $i^*(H^2(\bar{\mathcal X}))\subset H^2(D_\infty)$. The toric mirror map for $\bar{\mathcal X}$ is defined to be the map $y\mapsto q(y):=\exp\tau(y)$. Note that if $D_\infty\cdot d=0$, then the factor $\frac{1}{D_\infty+(D_\infty\cdot d)z}=1$ because 
\[
\frac{\prod_{a\leq D_\infty\cdot d-1}(\bar{D}_\infty+az)}{\prod_{a \leq 0}(\bar{D}_\infty+az)}=1
\]
under Convention \ref{conv-1}.

The toric mirror theorem for the pair is the following
\[
J_{(\bar{\mathcal X},D_\infty)}(\tau(y),z)=I_{(\bar{\mathcal X},D_\infty)}(y,z).
\]

We consider the coefficient of the $1/z^2$-term of $I_{(\bar{\mathcal X},D_\infty)}(y,z)$ that takes values in $H^0(D_\infty)$. Consider the expansion of the factor
\[
\frac{\prod_{a\leq 0, \langle a\rangle=\langle D_i\cdot d\rangle}(\bar{D}_i+az)}{\prod_{a \leq D_i\cdot d, \langle a\rangle=\langle D_i\cdot d\rangle}(\bar{D}_i+az)}.
\]
We write
\[
\frac{1}{\bar{D}_i+az}=\frac{1}{az(1+\bar{D}_i/az)}
\]
and expand it as $1/z$-series. For terms with values in $H^0(D_\infty)$, we need 
\[
v(d)=0, \text{that is, } \textbf{1}_{v(d)}=\textbf{1}\in H^0(\bar{\mathcal X}).
\]
Therefore, we must have
\[
D_i\cdot d\in \mathbb Z, \text{ for } i\in\{0,\ldots, m^\prime-1\}\cup\{\infty\}.
\]
For $i\in \{0,\ldots,m^\prime-1\}$, there are three possibilities:
\begin{enumerate}
\item $D_i\cdot d>0$: it contributes the factor $\frac{1}{(D_i\cdot d)!z^{D_i\cdot d}}$.
\item $D_i\cdot d=0$: it contributes $1$.
\item $D_i\cdot d<0$: it contributes a factor of $\bar{D}_i$, which is not allowed.
\end{enumerate}

Note that $\frac{1}{\bar{D}_\infty+(D_\infty\cdot d)z}$ contributes a factor $\frac{1}{(D_\infty\cdot d)z}$ if $D_\infty\cdot d>0$. Otherwise, the contribution is $1$. Hence the part of the $1/z^2$ that takes values in $H^0(D_\infty)$ is
\begin{enumerate}
\item $\sum_d \frac{y^d}{(\prod_{i=0}^{m^\prime-1}(D_i\cdot d)!)D_\infty\cdot d}$, if $D_\infty\cdot d>0$;
\item $\sum_d \frac{y^d}{(\prod_{i=0}^{m^\prime-1}(D_i\cdot d)!)}$, if $D_\infty\cdot d=0$.
\end{enumerate}
where the sum is over all $d$ such that 
\begin{enumerate}
\item  if $D_\infty\cdot d>0$, then $\left(\sum_{i=0}^{m^\prime-1}D_i\cdot d\right)+1=2$ and $D_i\cdot d\geq 0$ for all $i\in\{0,1,\ldots, m^\prime-1\}\cup\{\infty\}$;
\item if $D_\infty\cdot d=0$, then $\sum_{i=0}^{m^\prime-1}D_i\cdot d=2$ and $D_i\cdot d\geq 0$ for all $i\in\{0,1,\ldots, m^\prime-1\}$.
\end{enumerate}
For case $(1)$, 
we have $D_\infty\cdot d>0$, $D_j\cdot d=1$, for exactly one $j\in \{0,\ldots, m^\prime-1\}$ and $D_i\cdot d=0$ for all other $i$. Therefore, we have $j=i_0$ and $D_\infty\cdot d=1$. We simply have $y^{d_{\infty}}$. Note that there are actually two possibilities: if $0\leq i_0 \leq m-1$, then $d_\infty=e_{i_0}+e_\infty=\bar{\beta}^\prime\in H_2(\bar{\mathcal X},\mathbb Q)$; if $m\leq i_0\leq m^\prime-1$, then $d_\infty$ is not a class in  $ H_2(\bar{\mathcal X},\mathbb Q)$.

Case $(2)$ will not happen, since $\sum_{i=0}^{m^\prime-1}D_i\cdot d=2$ and $D_i\cdot d\geq 0$ for all $i=0,1,\ldots, m^\prime-1$ would imply $D_\infty\cdot d>0$. Indeed, since we can not have $D_i\cdot d=0$ for all but one $i$, we must be in the case that there are exactly two indices $j_1,j_2$ such that $D_{j_1}\cdot d=D_{j_2}\cdot d=1$ and $D_i\cdot d=0$, for $i\neq j_1,j_2$. Since $b_\infty$ lies in the half-space of $N_{\mathbb R}\oplus \mathbb R$ which is opposite to the half-space containing all the other vectors $b_0,\ldots,b_{m^\prime-1}$, we must have $\infty\in\{j_1,j_2\}$. Hence $D_\infty\cdot d>0$. It is a contradiction.

Therefore, by the mirror theorem for toric pairs, we obtain the following
\begin{proposition}\label{prop:rel-inv}
The following equality holds:
\[
y^{d_\infty}=\sum_{\alpha\in H_2^{\on{eff}}(\mathcal X)}\sum_{l\geq 0}\sum_{v_1,\ldots v_l\in \bbox(\Sigma)^{\age=1}}\frac{\prod_{i=1}^l \tau_{v_i}}{l!}\left\langle [\on{pt}]_{D_\infty},\textbf{1}_{\bar{v}_1},\ldots,\textbf{1}_{\bar{v}_l}\right\rangle_{0,1+l, \bar{\beta}^\prime+\alpha}^{(\bar{\mathcal X},D_\infty)}q^{ \bar{\beta}^\prime+\alpha},
\]
\end{proposition}

Therefore, in order to compute these invariants, we need to compute the mirror map.
\subsection{Toric mirror map}
Toric mirror map for $(\bar{\mathcal X},D_\infty)$ is given by the coefficient of the $1/z$-term of the $I$-function $I_{(\bar{\mathcal X},D_\infty)}(y,z)$ that takes values in $H^2(\mathcal X)$. We will see that it coincides with the toric mirror map for $\bar{\mathcal X}$.

\begin{lemma}
The product factor can be written as
\[
\prod_{i=0}^{m^\prime-1}\frac{\prod_{a\leq 0, \langle a\rangle=\langle D_i\cdot d\rangle}(\bar{D}_i+az)}{\prod_{a \leq D_i\cdot d,  \langle a\rangle=\langle D_i\cdot d\rangle}(\bar{D}_i+az)}=z^{-\hat{\rho}(\mathcal X)\cdot d-\age(v(d))}\prod_{i=0}^{m^\prime-1}\frac{\prod_{a\leq 0, \langle a\rangle=\langle D_i\cdot d\rangle}(\bar{D}_i/z+a)}{\prod_{a \leq D_i\cdot d,  \langle a\rangle=\langle D_i\cdot d\rangle}(\bar{D}_i/z+a)}
\]
\begin{proof}
This is a direct computation.
\end{proof}

\end{lemma}
Recall that we also have the factor
\[
\frac{1}{\bar{D}_\infty+(D_\infty\cdot d)z}=\frac{1}{(D_\infty\cdot d)z(1+\bar{D}_\infty/((D_\infty\cdot d)z))}, \text{ if } D_\infty\cdot d>0.
\]
We want to extract the coefficient of $z^{-1}$-term for the relative $I$-function $I_{(\bar{\mathcal X},D_\infty)}(y,z)$, there are two possibilities
\begin{enumerate}
\item if $D_\infty\cdot d>0$,  we need 
\[
\hat{\rho}(\mathcal X)\cdot d+\age(v(d))=0,\quad \text{ and } D_i\cdot d\geq 0, \text{ for all } 0\leq i\leq m^\prime-1
\]
It means that $D_i\cdot d=0$ for all $ 0\leq i\leq m^\prime-1$, but $D_\infty\cdot d=1$. Such a class $d\in \bar{\mathbb K}_{\on{eff}}$ does not exist.
\item $D_\infty\cdot d=0$, then we have $d\in \mathbb K_{\on{eff}}$. Therefore, we have 
\[
\sum_{i=0}^{m^\prime-1}D_i\cdot d=0.
\]
If $v(d)=0$, then,
\[
\#\{i:D_i\cdot d<0, 0\leq i\leq m-1\}=1,
\]
where $\#\{i:\cdots\}$ means the number of $i$ that satisfies the condition.
Hence the contribution to the coefficient of $z^{-1}$ of the relative $I$-function $I_{(\bar{\mathcal X},D_\infty)}(y,z)$ is
\[
\sum_{j=0}^{m-1}g_j(y)\bar{D}_j,
\]
where 
\begin{align}\label{g-j-smooth}
g_j(y):=\sum_{\substack{d:c_1(\mathcal X)\cdot d=0\\D_j\cdot d<0\\D_i\cdot d\geq 0, \forall i\neq j\\v(d)=0}}\frac{(-1)^{D_j\cdot d-1}(-D_j\cdot d-1)!}{\prod_{i\neq j}(D_i\cdot d)!}y^d,
\end{align}
for $0\leq j \leq m-1$.

If $v(d)\neq 0$, then 
\[
\#\{i:D_i\cdot d<0, 0\leq i\leq m-1\}=0, \text{ and } \age v(d)=1.
\]
By Assumption \ref{assumption-extra-vectors}, $v(d)=b_j$ for some $m\leq j\leq m^\prime-1$. Hence the contribution to the coefficient of $z^{-1}$ of the relative $I$-function $I_{(\bar{\mathcal X},D_\infty)}(y,z)$ is
\[
\sum_{j=m}^{m^\prime-1}g_j(y)\textbf{1}_{b_j},
\]
where
\begin{align}\label{g-j-orbi}
g_j(y):=\sum_{\substack{d:c_1(\mathcal X)\cdot d=0\\D_i\cdot d\geq 0, \forall i\\v(d)=b_j}}\prod_{i=0}^{m^\prime-1}\frac{\prod_{a\leq 0, \langle a\rangle=\langle D_i\cdot d\rangle}(a)}{\prod_{a \leq D_i\cdot d,  \langle a\rangle=\langle D_i\cdot d\rangle}(a)}y^d,
\end{align}
for $m\leq j \leq m^\prime-1$.
\end{enumerate}

The coefficient of the $1/z$-term in the relative $I$-function can be expressed as 
\[
\sum_{a=1}^{r^\prime}(\bar{p}_a\log y_a)+\bar{p}_{\infty}\log y_\infty+\sum_{j=0}^{m-1}g_j(y)\bar{D}_j+\sum_{j=m}^{m^\prime-1}g_j(y)\textbf{1}_{b_j}
\]

The coefficient of the $1/z$-term of the $J$-function is given by
\[
\sum_{a=1}^{r^\prime}\bar{p}_{a}\log q_a+\bar{p}_{\infty}\log q_\infty+\sum_{i=m}^{m^\prime-1}\tau_{b_j}{\textbf 1}_{b_j}
\]

The toric mirror map for $(\bar{\mathcal X},D_\infty)$ is obtained by comparing the $1/z$-coefficients of the $J$-function and $I$-function. Therefore, we have
\[
\log q_a=\log y_a+\sum_{j=0}^{m-1} m_{ja}g_j(y), \quad a\in\{1,\ldots, r^\prime\}\cup\{\infty\},
\]
\[
\tau_{b_j}=g_j(y), j\in\{m\ldots,m^\prime-1\}.
\]

Let us take a closer look at the case when $a=\infty$. When $\beta^\prime=\beta_{i_0}$ is a basic smooth disk class, we have $m_{i_0 \infty}=1, m_{\infty \infty}=1$ and $m_{j\infty}=0$ for $j\neq i_0,\infty$. We have
\[
\log q_\infty=\log y_\infty+g_{i_0(y)}
\]
When $\beta^\prime=\beta_{v_{j_0}}$ is a basic orbi-disk class, we have $m_{j_0\infty}=1, m_{\infty\infty}=1$ and $m_{j\infty}=0$ for $j\neq j_0,\infty$. In particular, it means that $m_{j\infty}=0$ for all $j\in\{0,\ldots,m-1\}$. We simply have
\[
\log q_\infty=\log y_\infty.
\]

We summarize our computation in the following proposition.
\begin{proposition}\label{prop:rel-toric-mirror-map}
The toric mirror map for $(\bar{\mathcal X},D_\infty)$ is of the following form. 
For $a\in\{1,\ldots, r^\prime\}$,
\[
\log q_a=\log y_a+\sum_{j=0}^{m-1} m_{ja}g_j(y).
\]
For $j\in\{m\ldots,m^\prime-1\}$,
\[
\tau_{b_j}=g_j(y).
\]
For $a=\infty$, when $\beta^\prime=\beta_{i_0}$ is a basic smooth disk class, we have 
\[
\log q_\infty=\log y_\infty+g_{i_0}(y);
\]
when $\beta^\prime=\beta_{v_{j_0}}$ is a basic orbi-disk class, we have 
\[
\log q_\infty=\log y_\infty.
\]
\end{proposition}

The toric mirror map for $\mathcal X$ was computed in \cite[Proposition 6.15]{CCLT} via a similar computation.
\begin{proposition}[\cite{CCLT}, Proposition 6.15]\label{prop:toric-mirror-map}
The toric mirror map for the toric Calabia--Yau orbifold $\mathcal X$ is given by
\[
\log q_a=\log y_a+\sum_{j=0}^{m-1} m_{ja}g_j(y), \quad a\in\{1,\ldots, r^\prime\},
\]
\[
\tau_{b_j}=g_j(y), j\in\{m\ldots,m^\prime-1\},
\]
where the functions $g_j(y)$ are defined in (\ref{g-j-smooth}) and (\ref{g-j-orbi}). 
\end{proposition}

By Proposition \ref{prop:rel-inv}, we can write the following generating function for relative invariants as 
\[
\sum_{\alpha\in H_2^{\on{eff}}(\mathcal X)}\sum_{l\geq 0}\sum_{v_1,\ldots v_l\in \bbox(\Sigma)^{\age=1}}\frac{\prod_{i=1}^l \tau_{v_i}}{l!}\left\langle [\on{pt}]_{D_\infty},\textbf{1}_{\bar{v}_1},\ldots,\textbf{1}_{\bar{v}_l}\right\rangle_{0,1+l, \bar{\beta}^\prime+\alpha}^{(\bar{\mathcal X},D_\infty)}q^{\alpha}=y^{d_\infty}q^{-\bar{\beta}^\prime}
\]

We can explicitly compute these relative invariants. The result gives an enumerative meaning of the inverse of the toric mirror map of $\mathcal X$ in terms of relative Gromov--Witten invariants.

\begin{theorem}\label{thm:rel-smooth}
If $\beta^\prime=\beta_{i_0}$ is a basic smooth disk class for some $i_0\in\{0,1,\ldots,m-1\}$, then we have
\begin{align*}
1+\delta_{i_0}^{\on{rel}}:=&\sum_{\alpha\in H_2^{\on{eff}}(\mathcal X)}\sum_{l\geq 0}\sum_{v_1,\ldots v_l\in \bbox(\Sigma)^{\age=1}}\frac{\prod_{i=1}^l \tau_{v_i}}{l!}\left\langle [\on{pt}]_{D_\infty},\textbf{1}_{\bar{v}_1},\ldots,\textbf{1}_{\bar{v}_l}\right\rangle_{0,1+l, \bar{\beta}^\prime+\alpha}^{(\bar{\mathcal X},D_\infty)}q^{\alpha}\\
=&\exp(-g_{i_0}(y(q,\tau))),
\end{align*}
where $y=y(q,\tau)$ is the inverse of the toric mirror map of $\mathcal X$ in Proposition \ref{prop:toric-mirror-map}.
\end{theorem}

\begin{proof}
By Proposition \ref{prop:rel-inv}, it remains to compute $y^{d_\infty}q^{-\bar{\beta}^\prime}$.
Since $\beta^\prime=\beta_{i_0}$ is a basic smooth disk class, we have $\bar{\beta}^\prime=d_\infty$. Furthermore, we have
\[
D_\infty=p_\infty; \quad p_\infty\cdot d_\infty=1;\quad  \bar{D}_i\cdot d_\infty=D_i\cdot d_\infty, \forall i;\quad \bar{p}_a\cdot d_\infty=p_a\cdot d_\infty, \forall a. 
\]
Therefore,
\begin{align*}
&\log q^{d_\infty}-\log y^{d_\infty}\\
=&\sum_{a=1}^{r^\prime}(\bar{p}_a\cdot d_\infty)\log q_a+(\bar{p}_\infty\cdot d_\infty)\log q_\infty-\log y^{d_\infty}.\\
=& \sum_{a=1}^{r^\prime}(\bar{p}_a\cdot d_\infty)\left(\log y_a+\sum_{j=0}^{m-1} m_{ja}g_j(y)\right)+\left(\log y_\infty+g_{i_0}(y)\right)-\log y^{d_\infty}\\
=&\left(\sum_{j=0}^{m-1} \left(\sum_{a=1}^{r^\prime}m_{ja}(\bar{p}_a\cdot d_\infty)\right)g_j(y)\right)+g_{i_0}(y)\\
=&\left(\sum_{j=0}^{m-1} \left(D_j\cdot d_\infty-m_{j\infty}(\bar{p}_\infty\cdot d_\infty)\right)g_j(y)\right)+g_{i_0}(y)\\
=&\left(\sum_{j=0}^{m-1} \left(m_{j\infty}-m_{j\infty}\right)g_j(y)\right)+g_{i_0}(y)\\
=& g_{i_0}(y),
\end{align*}
where the third line follows from Proposition \ref{prop:rel-toric-mirror-map} and the rest of the lines follow from our toric setting.
Then the theorem follows.
\end{proof}

\begin{theorem}\label{thm:rel-orbi}
If $\beta^\prime=\beta_{v_{j_0}}$ is a basic orbi-disk class corresponding to a box element $v_{j_0}\in \bbox(\Sigma)^{\age=1}$ for some $j_0\in\{m,\ldots, m^\prime-1\}$, then we have
\begin{align*}
\tau_{v_{j_0}}+\delta_{v_{j_0}}^{\on{rel}}:=&\sum_{\alpha\in H_2^{\on{eff}}(\mathcal X)}\sum_{l\geq 0}\sum_{v_1,\ldots v_l\in \bbox(\Sigma)^{\age=1}}\frac{\prod_{i=1}^l \tau_{v_i}}{l!}\left\langle [\on{pt}]_{D_\infty},\textbf{1}_{\bar{v}_1},\ldots,\textbf{1}_{\bar{v}_l}\right\rangle_{0,1+l, \bar{\beta}^\prime+\alpha}^{(\bar{\mathcal X},D_\infty)}q^{\alpha}\\
=& y^{D_{j_0}^\vee}\exp\left(-\sum_{i\not\in I_{j_0}}c_{j_0 i}g_{i}(y(q,\tau))\right),
\end{align*}
where $y=y(q,\tau)$ is the inverse of the toric mirror map of $\mathcal X$ in Proposition \ref{prop:toric-mirror-map}; $D_{j_0}^\vee$ is the class defined in (\ref{D-dual}); $I_{j_0}\in \mathcal A$ is the anticone of the minimal cone containing $b_{j_0}=\sum_{i\not\in I_{j_0}}c_{j_0i}b_i$.

\end{theorem}

\begin{proof}
Since $\beta^\prime=\beta_{v_{j_0}}$ is a basic orbi-disk class, we have
\[
\bar{\beta}^\prime=\left(\sum_{i\not\in I_{j_0}}c_{j_0 i}e_i\right)+e_{\infty},\quad \text{and } d_\infty=e_{j_0}+e_\infty.
\]
Therefore, by the definition of $D_{j_0}^\vee$ in (\ref{D-dual}),
\[
d_\infty-\bar{\beta}^\prime=D_{j_0}^\vee\in\mathbb K_{\on{eff}}.
\]

Hence, we have
\[
y^{d_\infty}q^{-\bar{\beta}^\prime}=y^{D_{j_0}^\vee}y^{\bar{\beta}^\prime}q^{-\bar{\beta}^\prime}.
\]
Now it remains to compute $y^{\bar{\beta}^\prime}q^{-\bar{\beta}^\prime}$. We have
\begin{align*}
&\log q^{\bar{\beta}^\prime}-\log y^{\bar{\beta}^\prime}\\
=&\sum_{a=1}^{r^\prime}(\bar{p}_a\cdot \bar{\beta}^\prime)\log q_a+(\bar{p}_\infty\cdot \bar{\beta}^\prime)\log q_\infty-\log y^{\bar{\beta}^\prime}\\
=& \sum_{a=1}^{r^\prime}(\bar{p}_a\cdot \bar{\beta}^\prime)\left(\log y_a+\sum_{i=0}^{m-1}m_{ia}g_i(y)\right)+(\bar{p}_\infty\cdot \bar{\beta}^\prime)\log y_\infty-\log y^{\bar{\beta}^\prime}\\
=& \sum_{a=1}^{r^\prime}(\bar{p}_a\cdot \bar{\beta}^\prime)(\log y_a)+ \sum_{a=1}^{r^\prime}(\bar{p}_a\cdot \bar{\beta}^\prime)\left(\sum_{i=0}^{m-1}m_{ia}g_i(y)\right)+(\bar{p}_\infty\cdot \bar{\beta}^\prime)\log y_\infty-\log y^{\bar{\beta}^\prime}\\
=&  \sum_{i=0}^{m-1}\left(\sum_{a=1}^{r^\prime} m_{ia}(\bar{p}_a\cdot \bar{\beta}^\prime)g_i(y)\right)\\
=&\sum_{i=0}^{m-1}\left((\bar{D}_i\cdot\bar{\beta}^\prime) g_i(y)\right)\\
=& \sum_{i\not\in I_{j_0}}c_{j_0 i}g_{i}(y).
\end{align*}
where the third line follows from Proposition \ref{prop:rel-toric-mirror-map}, the sixth line follows from the fact that $m_{i\infty}=0$ for $i\in\{0,\ldots,m-1\}$, and the rest of the lines follow from our toric setting.
Then the theorem follows.
\end{proof}


\subsection{Relative invariants as disk countings}\label{sec:open-closed-duality}

The formulas for the genus zero open orbifold Gromov--Witten invariants was proved in \cite{CCLT} by relating open invariants of toric Calabi--Yau orbifold $\mathcal X$ with certain (absolute) closed invariants of the toric compactification $\bar {\mathcal X}$. Then the corresponding closed invariants were computed via toric mirror theorem. 
\begin{theorem}[\cite{CCLT}, Theorem 6.19]\label{thm:open-smooth}
If $\beta^\prime=\beta_{i_0}$ is a basic smooth disk class for some $i_0\in\{0,1,\ldots,m-1\}$, then we have
\begin{align*}
1+\delta_{i_0}^{\on{open}}=\exp(-g_{i_0}(y(q,\tau))),
\end{align*}
where $y=y(q,\tau)$ is the inverse of the toric mirror map of $\mathcal X$ in Proposition \ref{prop:toric-mirror-map}.
\end{theorem}

\begin{theorem}[\cite{CCLT}, Theorem 6.20]\label{thm:open-orbi}
If $\beta^\prime=\beta_{v_{j_0}}$ is a basic orbi-disk class corresponding to a box element $v_{j_0}\in \bbox(\Sigma)^{\age=1}$ for some $j_0\in\{m,\ldots, m^\prime-1\}$, then we have
\begin{align*}
\tau_{v_{j_0}}+\delta_{v_{j_0}}^{\on{open}}= y^{D_{j_0}^\vee}\exp\left(-\sum_{i\not\in I_{j_0}}c_{j_0 i}g_{i}(y(q,\tau))\right),
\end{align*}
where $y=y(q,\tau)$ is the inverse of the toric mirror map of $\mathcal X$ in Proposition \ref{prop:toric-mirror-map}; $D_{j_0}^\vee$ is the class defined in (\ref{D-dual}); $I_{j_0}\in \mathcal A$ is the anticone of the minimal cone containing $b_{j_0}=\sum_{i\not\in I_{j_0}}c_{j_0i}b_i$.
\end{theorem}

As a direct consequence of Theorem \ref{thm:rel-smooth}, Theorem \ref{thm:rel-orbi}, Theorem \ref{thm:open-smooth} and Theorem \ref{thm:open-orbi}, we have the following equality between open invariants and relative invariants

\begin{theorem}\label{thm:open=rel}
For $i_0\in \{0,1,\ldots,m-1\}$, we have
\[
1+\delta_{i_0}^{\on{open}}=1+\delta_{i_0}^{\on{rel}}.
\]
For $j_0\in\{m+\ldots,m^\prime-1\}$, we have
\[
\tau_{v_{j_0}}+\delta_{v_{j_0}}^{\on{open}}=\tau_{v_{j_0}}+\delta_{v_{j_0}}^{\on{rel}}.
\]
In other words, the following open invariants and relative invariants are equal:
\[
n_{1,l,\beta^\prime+\alpha}^{\mathcal X}( [\on{pt}]_L,\textbf{1}_{v_1},\ldots,\textbf{1}_{v_l})=\left\langle [\on{pt}]_{D_\infty},\textbf{1}_{\bar{v}_1},\ldots,\textbf{1}_{\bar{v}_l}\right\rangle_{0,1+l, \bar{\beta}^\prime+\alpha}^{(\bar{\mathcal X},D_\infty)},
\]
where $\bar{\beta}^\prime=\beta^\prime+\beta_\infty$ and $\beta^\prime\in \pi_2(\mathcal X,L)$ is either a smooth disk class or an orbi-disk class.
\end{theorem}

\begin{rmk}
Geometrically, the equality between open invariants and relative invariants can be understood as follows. Open invariants of $\mathcal X$ are considered as virtual counts of orbi-disks in $\mathcal X$. On the other hand, relative invariants of $(\bar{\mathcal X},D_\infty)$ with one relative marking can be viewed as virtual counts of orbi-curves in $\bar{\mathcal X}$ that meet with the divisor $D_\infty$ at one point. Hence these relative invariants can also be understood as virtual counts of orbi-disks in $\mathcal X=\bar{\mathcal X}\setminus D_\infty$, the complement of the divisor $D_\infty$ of the toric compactification $\bar{\mathcal X}$ of $\mathcal X$. Therefore, the relative invariants in Theorem \ref{thm:open=rel} can be viewed as an algebro-geometric version of counting disks in $\mathcal X$. This heuristic view can be seen in several literatures, see for example \cite{GHK}. Hence, Theorem \ref{thm:open=rel} shows that the symplectic and algebraic ways of counting disks coincide for toric Calabi-Yau orbifolds. 
\end{rmk}

\begin{rmk}
Note that open invariants with different disk classes correspond to relative invariants of different toric compactification $\bar{\mathcal X}$. There is a one-to-one correspondence between basic disk classes and toric compactifications.
\end{rmk}

\section{Mirror construction}

The instanton-corrected SYZ mirror construction for toric Calabi--Yau orbifolds was carried out in \cite[Section 5]{CCLT}. The instanton corrections were obtained by taking a family version of Fourier transformations of generating functions of genus $0$ open orbifold Gromov--Witten invariants which count orbifold disks with Chen-Weil Maslov index $2$. By Theorem \ref{thm:open=rel}, we replace open Gromov--Witten invariants with relative Gromov--Witten invariants (with one relative marking).

Recall that, $\mathcal X$ is a toric Calabi--Yau orbifold. Let $\underline{v}\in M$ be the vector such that $(\underline{v},b_i)=1$ for all stacky vectors $b_i$. The vector $\underline{v}\in M$ corresponds to a holomorphic function $\omega:\mathcal X\rightarrow \mathbb C$. The complexified extended K\"ahler moduli space of $\mathcal X$ is defined as 
\[
\mathcal M_{K}(\mathcal X):=\left(\tilde{C}_{\mathcal X}+\sqrt{-1}H^2(\mathcal X,\mathbb R)\right)/H^2(\mathcal X,\mathbb Z)+\sum_{j=m}^{m^\prime-1}\mathbb C D_j,
\]
where the elements of $\mathcal M_K(\mathcal X)$ are represented by complexified extended K\"ahler class
\[
\omega^{\mathbb C}=\omega+\sqrt{-1}\omega^\prime+\sum_{j=m}^{m^\prime-1}\tau_j D_j
\]
with $\omega\in C_{\mathcal X}$, $\omega^\prime\in H^2(\mathcal X,\mathbb R)$ and $\tau_j\in \mathbb C$. We have
\[
q_a=\exp\left(-2\pi\int_{\gamma_a}(\omega+\sqrt{-1}\omega^\prime)\right), \quad a=1,\ldots, r^\prime,
\]
where $\{\gamma_1,\ldots, \gamma_{r^\prime}\}$ is the integral basis of $H_2(\mathcal X,\mathbb Z)$ in Section \ref{sec:def-toric}.

The SYZ mirror of $\mathcal X$ equipped with a Gross fibration $\mu:\mathcal X \rightarrow B$ is given in \cite[Theorem 1.1]{CCLT}.
Consider $C_{i},C_{v_j}\in \mathbb C$ with the following relations
\[
\prod_{i=0}^{m-1}C_i^{m_{ia}}=q_a, \quad a=1,\ldots, r^\prime,
\]
\[
\prod_{i=0}^{m-1}C_i^{m_{ia}}\prod_{j=m}^{m^\prime-1}C_{v_j}^{m_{ja}}=\prod_{j=m}^{m^\prime-1}(q^{D_j^\vee})^{-m_{ja}}, \quad a=r^\prime+1,\ldots, r,
\]
where $q^{D_j^\vee}=\prod_{a=1}^{r^\prime}q_a^{p_a\cdot D_j^\vee}$. One can define the following function.
\[
G_{(q,\tau)}(z_1,\ldots,z_{n-1})=\sum_{i=0}^{m-1}C_i(1+\delta_i^{\on{open}})z^{b_i}+\sum_{j=m}^{m^\prime-1}C_{v_j}(\tau_{v_j}+\delta_{v_j}^{\on{open}})z^{v_j},
\]
The the SYZ mirror of $\mathcal X$ with a hypersurface removed is given by
\[
\check{\mathcal X}_{q,\tau}=\{(u,v,z_1,\ldots,z_{n-1})\in \mathbb C^2\times (\mathbb C^\times)^{n-1}| uv=G_{(q,\tau)}(z_1,\ldots,z_{n-1})\}.
\]
The SYZ mirror of $\mathcal X$ without a hypersurface removed is given by the Landau--Ginzburg model $(\check{\mathcal X},W)$ where $W:\check{\mathcal X}\rightarrow \mathbb C$ is the holomorphic function $W:=u$.

By Theorem \ref{thm:open=rel}, the SYZ mirror can be written in terms of relative invariants instead of open invariants

\begin{theorem}\label{thm:SYZ-mirror}
Let $\mathcal X$ be a toric Calabi--Yau orbifold equipped with the Gross fibration, the SYZ mirror of $\mathcal X$ (with a hypersurface removed) is the family of non-compact Calabi--Yau
\[
\check{\mathcal X}_{q,\tau}=\{(u,v,z_1,\ldots,z_{n-1})\in \mathbb C^2\times (\mathbb C^\times)^{n-1}| uv=G_{(q,\tau)}(z_1,\ldots,z_{n-1})\},
\]
where
\[
G_{(q,\tau)}(z_1,\ldots,z_{n-1})=\sum_{i=0}^{m-1}C_i(1+\delta_i^{\on{rel}})z^{b_i}+\sum_{j=m}^{m^\prime-1}C_{v_j}(\tau_{v_j}+\delta_{v_j}^{\on{rel}})z^{v_j},
\]

The SYZ mirror of $\mathcal X$ without removing a hypersurface is given by a Landau--Ginzburg model $(\check{\mathcal X},W)$, where $W:\check{\mathcal X}\rightarrow \mathbb C$ is the holomorphic function $W:=u$.
\end{theorem}

\begin{remark}
As mentioned in \cite{GHK}, heuristically, holomorphic disks in $X\setminus D$ can be approximated by proper rational curves meeting the divisor $D$ in a single point. Therefore, one can consider relative invariants of $(X,D)$ as an algebro-geometric analogue of a Maslov index two holomorphic disk in $X$. Since relative invariants are defined in much more general context than open invariants, Theorem \ref{thm:SYZ-mirror} suggests that  relative invariants may play crucial roles in mirror construction in general. 
\end{remark}

\begin{remark}
Let $\mathcal M_{\mathbb C}(\check{\mathcal X})$ be the complex moduli space of the mirror $\check{\mathcal X}$ defined in \cite[Section 7.1.2]{CCLT}. Then the SYZ map $\mathcal F^{\on{SYZ}}(q,\tau)$ in \cite[Definition 7.1]{CCLT} can be defined using relative invariants. Then the open mirror theorems in \cite[Section 7.2]{CCLT} can be stated in terms of relative invariants. In other words, the SYZ map, defined by replacing open invariants with relative invariants, is inverse to the toric mirror map. The SYZ map also coincides with the inverse of the mirror map defined in terms of period integrals. Therefore, this gives an enumerative meaning of the period integrals.
\end{remark}

\appendix

\section{Genus zero relative/orbifold correspondence with stacky targets}\label{sec:stacky-rel=orb}


In this section, we briefly describe the proof of the equality between genus zero relative and orbifold invariants with stacky targets. More details will be appeared in the forthcoming paper with H.-H. Tseng in \cite{TY19}, where we will prove the relation between relative and orbifold invariants with stacky targets in all genera. The result generalizes the result of \cite{TY18} to stacky targets. Higher genus case requires some extra works as it requires certain polynomiality, but the proof for genus zero invariants is a straight adaptation of \cite[Section 5]{TY18}. 

Let $\mathcal X$ be a smooth proper Deligne--Mumford stack and $\mathcal D\subset \mathcal X$ a smooth irreducible divisor. We study the relation between genus zero relative invariants of the pair $(\mathcal X,\mathcal D)$ and genus zero orbifold invariants of the $r$-th root stack $\mathcal X_{\mathcal D,r}$ of $\mathcal X$ along the divisor $\mathcal D$. The proof can be divided into the following two steps.

\begin{enumerate}

\item[Step 1:] reduction to local models. It is observed in \cite[Section 2.3]{TY18a} and \cite[Section 3]{TY18} that the comparison between relative invariants and orbifold invariants is local over the divisor. Therefore, it is sufficient to compare the relative local models. 

The reduction is achieved by applying degeneration formula to both $\mathcal X_{\mathcal D,r}$ and $(\mathcal X,\mathcal D)$. 
Define $\mathcal Y:=\mathbb P(\mathcal O\oplus \mathcal N_{\mathcal D})$, where $\mathcal N$ is the normal bundle of $\mathcal D\subset \mathcal X$. The $r$-th root stack of $\mathcal Y$ along the zero section $\mathcal D_0\subset \mathcal Y$ is denoted by $\mathcal Y_{\mathcal D_0,r}$.
The degeneration formula allows one to write the orbifold invariants of $\mathcal X_{\mathcal D,r}$ as a sum of products of orbifold-relative invariants of $(\mathcal Y_{\mathcal D_0,r},\mathcal D_\infty)$ and relative invariants of $(\mathcal X,\mathcal D)$. On the other hand, the degeneration formula allows one to write the relative invariants of the pair $(\mathcal X,\mathcal D)$ as a sum of products of relative invariants of $(\mathcal Y,\mathcal D_0\cup \mathcal D_\infty)$ and the relative invariants of $(\mathcal X,\mathcal D)$. Notice that the two summations range over the same intersection profiles along $\mathcal D$, the comparison reduces to the comparison between the invariants of the following pairs
\[
(\mathcal Y_{\mathcal D_0,r},\mathcal D_\infty) \quad, \text{and} \quad (\mathcal Y,\mathcal D_0\cup \mathcal D_\infty),
\] 
which are called relative local models in \cite{TY18}.
\item[Step 2:] virtual localization. As discussed in \cite{TY18}, it is sufficient to generalize \cite[Lemma 5.2]{TY18} to the case when the divisor $\mathcal D$ is stacky. The proof for the case when the divisor $\mathcal D$ is stacky follows from the same path. The natural fiberwise $\mathbb C^*$-action on $\mathcal Y$ induces $\mathbb C^*$-actions on the relevant moduli spaces of stable maps to relative local models. \cite[Lemma 5.2]{TY18} is proved via ($\mathbb C^*$)-virtual localization. Localization formula can be written in terms of decorated graphs. The edge contributions will be trivial as explained in \cite{TY18}. By comparing vertex contributions, one yields the desired formula. 

\end{enumerate}

\begin{remark}
The equality between relative and orbifold invariants is actually on the level of virtual cycles as the two steps described above are on cycle level. We refer the readers to \cite[Section 6]{FWY} for more details of how to adapt the above argument to obtain an equality between virtual cycles. A different approach is presented in \cite{ACW}. 
\end{remark}

\end{document}